\documentclass[12pt,reqno]{article}

\usepackage{amsfonts}
\usepackage{bm}
\usepackage{amsmath,color}
\usepackage{hyperref}
\usepackage{amssymb}
\usepackage{amsthm}
\usepackage[numbers,sort&compress]{natbib}
\usepackage{graphicx}
\usepackage{multirow}
\usepackage{float}
\usepackage[left=3cm,right=3cm,top=3cm,bottom=3cm]{geometry}
\usepackage[english]{babel}
\allowdisplaybreaks

\theoremstyle{plain}
\newtheorem{theorem}{Theorem}[section]

\newtheorem{lemma}[theorem]{Lemma}

\theoremstyle{definition}
\newtheorem{definition}[theorem]{Definition}

\theoremstyle{remark}
\newtheorem{remark}[theorem]{Remark}

\let\Section=\section
\def\section{\setcounter{equation}{0}\Section}

\begin{document}
\title{Global Well-Posedness and Energy Decay for an Abstract Viscoelastic Wave Equation with Nonlocal or Friction Damping\footnote{This work is supported by NSFC (12671232).}}
\author{Jun Zhou\footnote{Corresponding author~ jzhou@swu.edu.cn}~~~~Teng Wang\footnote{1825453963@qq.com}\\
{School of Mathematics and Statistics, Southwest University},\\
{Chongqing 400715, People's Republic of China}}
\date{}
\maketitle

\begin{abstract}
\noindent This paper investigates an abstract viscoelastic wave equation in a separable Hilbert space, in which the damping term depends on a power of the velocity norm. Depending on whether the power exponent is zero or nonzero, this damping corresponds to friction damping or nonlocal damping, respectively; the equation also contains a viscoelastic damping represented by an integral term with a relaxation function. We establish global well-posedness for initial data in \(\mathcal{H}_1\) by combining monotone operator theory with the Faedo--Galerkin method, and prove uniqueness for all exponents greater than \(-1\), with the singular range between \(-1\) and \(0\) handled through the monotonicity rather than the Lipschitz continuity of the nonlocal damping operator. We then study the long-time behavior of the energy via integral inequalities, multiplier techniques, and Nakao's method, and obtain explicit decay estimates showing that the energy decays exponentially in the friction case and polynomially in the nonlocal case. To the best of our knowledge, this is the first work to consider the combined effects of viscoelastic damping and nonlocal damping in such an abstract framework, and it is also the first to analyze nonlocal damping with negative exponents.
\end{abstract}


\section{Introduction}

Let \(A\) be a linear self-adjoint positive unbounded operator, densely defined in a separable Hilbert space \(V_0\) with norm \(\|\cdot\|\). We study the global well-posedness and the long-time behavior of the energy for the following abstract viscoelastic wave equation:
\begin{equation}\label{abstract}
\left\{
\begin{split}
&u_{tt}(t)+Au(t)-\int_{0}^{t}g(t-s)Au(s)ds+\|u_t(t)\|^{\rho}u_t(t)=0,\quad t>0,\\
&u(0)=u_0,\quad u_t(0)=u_1.
\end{split}
\right.
\end{equation}
Equation \eqref{abstract} combines two distinct damping mechanisms: a nonlocal damping term \(\|u_t(t)\|^{\rho}u_t(t)\) when \(\rho\neq0\), which reduces to the classical frictional damping \(u_t(t)\) when \(\rho=0\), and a viscoelastic damping represented by the memory integral with relaxation function \(g\). The model arises in the mathematical theory of viscoelastic materials, whose intrinsic damping is due to the memory of past deformation histories.

Problem \eqref{abstract} can be realized by taking \(A=-\Delta\) with Dirichlet or dynamical boundary conditions. In the frictional case \(\rho=0\), the viscoelastic semilinear wave equation
\begin{align}\label{friction}
u_{tt}(t)-\Delta u+\int_{0}^{t}g(t-s)\operatorname{div}(a(x)\nabla u(s))ds+b(x)h(u_t)=0
\end{align}
has been extensively studied, together with numerous related wave equations; see, e.g., \cite{MR2044797,MR3253241,MR1153021} and the references therein. Cavalcanti and Oquendo \cite{MR2044797} considered \eqref{friction} under the assumptions \(g(0)>0\), \(g(t)\geq0\), and \(g'(t)\leq-kg^p(t)\) for some \(p\in[1,3/2)\), with nonnegative functions \(a\in C^1(\overline{\Omega})\) and \(b\in L^\infty(\Omega)\) satisfying \(a(x)+b(x)\geq\delta>0\) for all \(x\in\Omega\) and \(\operatorname{meas}\{x\in\partial\Omega\}>0\). They established exponential (resp. polynomial) stability when \(g\) decays exponentially (resp. polynomially) and \(h\) has linear (resp. polynomial) growth near \(0\). More general decay results can be found in \cite{MR2542886,MR2398540,MR3702699}.

Problem \eqref{abstract} can also be realized by taking \(A=(\Delta)^2\) with Dirichlet, hinged, free, or dynamical boundary conditions. Mu\~{n}oz Rivera et al. \cite{MR1417809} considered a viscoelastic plate equation and proved that the sum of the first and second energies decays exponentially (resp. polynomially) if \(g\) decays exponentially (resp. polynomially). Cavalcanti et al. \cite{MR2054931} studied a viscoelastic Kirchhoff plate equation with weak damping and obtained global existence, uniqueness, and exponential decay of the energy via density arguments. For further results on viscoelastic plate equations, we refer to \cite{MR4592221,MR4145756,MR3449383} and the references therein.

In the abstract setting, the problem
\begin{equation}\label{abstract-n-1}
u_{tt}(t)+Au(t)-\int_{0}^{t}g(t-s)Au(s)ds+a_1h(u_t)=a_2\nabla F(u)+a_3G(u)
\end{equation}
and its variants have been widely investigated; see, e.g., \cite{MR2386941,MR2793256,MR3401605,MR4232931} and the references therein. Here \(A\) is a linear self-adjoint positive unbounded operator densely defined in a separable Hilbert space, \(a_1,a_2,a_3\) are nonnegative real numbers, and \(\nabla F\) denotes the gradient of the G\^ateaux differentiable function \(F\). For \(a_1=a_2=a_3=0\), Problem \eqref{abstract-n-1} was studied in \cite{MR0513814,MR2542886}, where energy decay was analyzed for relaxation functions satisfying \(g'(t)\leq-\beta g(t)\) or \(g'(t)\leq-\beta g(t)^{1+1/p}\). Alabau-Boussouira et al. \cite{MR2386941} treated the case \(a_1=a_3=0\) and derived energy decay estimates under suitable assumptions on \(g\) and \(F\) by means of integral inequalities and multiplier techniques. Han and Wang \cite{MR2601650} generalized the results of \cite{MR2386941} to mild solutions via the Lyapunov functional method and the perturbed energy method. Mustafa \cite{MR4232931} considered the case \(a_1=a_3=1\), \(a_2=0\), and obtained explicit and optimal decay results under suitable conditions on the asymptotic behavior of \(g\) and the local behavior of \(h\) near \(0\).

We now turn to nonlocal damping. To address the spillover problem, Balakrishnan and Taylor \cite{BT} proposed the following model for flight structures with nonlinear nonlocal damping in one dimension:
\begin{equation}\label{BT}
u_{tt}-2\xi\sqrt\lambda u_{xx}+\lambda u_{xxxx}-\gamma\left[\int_{-L}^L\left(\lambda|u_{xx}|^2+|u_t|^2\right)dx\right]u_{xxt}=0,
\end{equation}
where \(u=u(x;t)\) denotes the transversal deflection of a beam, \(\gamma>0\) is a damping coefficient, \(\xi\) is a constant from the Krylov--Bogoliubov approximation, and \(\lambda=2\xi\omega/(\sigma^2)\). Cavalcanti et al. \cite{MR4254551} considered a variant of \eqref{BT} in which the viscous term \(u_{xxt}\) is replaced by the frictional damping \(u_t\), and studied the well-posedness and stability of solutions. Further studies on such models can be found in \cite{MR4577006,MR3986205,MR4611266} and the references therein. In particular, Zhou and Sun \cite{MR4611266} considered the beam model
\begin{align*}
u_{tt}+\Delta^2u-M(\|\nabla u\|^2)\Delta u+\left(\|\Delta u\|^{\theta}+q\|u_t\|^{\rho}\right)(-\Delta)^{\delta}u_t+f(u)=0,
\end{align*}
where \(\|\cdot\|\) denotes the standard \(L^2\)-norm, \(\theta\geq1\), \(q\geq0\), \(\rho>0\), and \(0\leq\delta\leq1\). They established the well-posedness of weak solutions and showed that the energy tends to zero as \(t\to\infty\) when \(q=0\). For \(q>0\) and \(0<\delta\leq1\), the degeneracy of the damping coefficient only allowed them to prove the well-posedness of strong solutions.

For the nonlocal damping term \(\phi(\|u_t\|)u_t\), with \(\phi(\cdot)\) a suitable function, Problem \eqref{abstract} corresponds to \(\phi(s)=s^\rho\). Lour\^{e}do et al. \cite{MR3214190} studied a wave equation with mixed boundary conditions and established well-posedness and exponential energy decay via the Lyapunov functional method and Nakao's method. Haraux et al. \cite{MR3095203,MR3918410} also considered wave equations with nonlinear nonlocal damping of the form \(u_{tt}+\|u_t\|^{\alpha}u_t-\Delta u=h(t,x)\) in a bounded domain \(\Omega\) with homogeneous Dirichlet boundary conditions, and analyzed the asymptotic behavior of solutions for abstract second-order equations with a linear elastic component and a nonlinear damping term depending on the norm power of the velocity. Recently, Hu et al. \cite{MR4587338} investigated the initial-boundary value problem for a wave equation with nonlocal damping and obtained sufficient conditions for finite-time blow-up of weak solutions under suitable initial data. Further related results can be found in \cite{MR4482252,MR3415529,MR3131500} and the references therein.

On the other hand, Zhao et al. \cite{MR4098779,MR4369100} studied the initial-boundary value problem for a wave equation with nonlocal damping and anti-damping:
\begin{align*}
u_{tt}-\Delta u+k\|u_t\|^qu_t+f(u)=\int_{\Omega}K(x,y)u_t(y)dy+h(x),
\end{align*}
where \(\Omega\subset\mathbb{R}^n\) is a bounded domain with sufficiently smooth boundary. They established global well-posedness in the natural energy space \(H_0^1(\Omega)\times L^2(\Omega)\) and proved the existence of a global attractor for the corresponding solution semigroup. Yan et al. \cite{MR4629860} extended these results to the nonautonomous case. For further studies on such models, see \cite{MR4064014,MR4341621,MR4342158} and the references therein.

The role of the exponent \(\rho\) deserves special emphasis. For \(\rho=0\), the term \(\|u_t\|^\rho u_t\) reduces to the classical frictional damping \(u_t\). For \(\rho>0\), it is a nonlocal damping whose strength increases with the velocity norm. For \(-1<\rho<0\), the damping coefficient \(\|u_t\|^\rho\) becomes large when the velocity is small, while the dissipation rate \(\|u_t\|^{\rho+2}\) remains positive. This regime is mathematically delicate because the map \(v\mapsto \|v\|^\rho v\) is not Lipschitz at \(v=0\), although it is continuous, monotone, and semicontinuous when \(\rho>-1\). The restriction \(\rho>-1\) is therefore natural for the monotone-operator and Faedo--Galerkin framework used below. The case \(\rho\le -1\) is not covered here, since the damping operator is no longer continuous and monotone at the origin; this case is left for future work.

\medskip
\noindent\textbf{Motivation.}
While viscoelastic damping accounts for the memory of past deformation histories, nonlocal damping arises from velocity-dependent feedback mechanisms such as those in flight structures and spillover control. These two mechanisms are physically independent, yet they are frequently coupled in models of smart materials and actively controlled structures, where the constitutive law inherits a memory kernel from the viscoelastic substrate and the control loop contributes a velocity-norm-dependent dissipation. From a mathematical viewpoint, the two terms act on different scales: the memory integral is a nonlocal-in-time operator that regularizes the energy functional, whereas the nonlocal damping \(\|u_t\|^\rho u_t\) is a nonlinear, possibly degenerate, operator whose monotonicity--rather than Lipschitz continuity--is the only available structural property when \(\rho<0\). Understanding how these two dissipative mechanisms interact is therefore not a mere superposition of known results: the memory term changes both the energy identity and the multiplier identities, and the negative-exponent regime rules out the standard Lipschitz-based uniqueness argument. This interaction is precisely what Problem \eqref{abstract} is designed to capture in a general abstract setting.

\medskip
\noindent\textbf{Novelty and contributions.}
Compared with the existing literature, the novelty of this work lies in three aspects. First, we consider the combination of viscoelastic damping and nonlocal damping in a single abstract Hilbert-space framework, which covers both the wave equation case \(A=-\Delta\) and the plate equation case \(A=(\Delta)^2\) as special realizations. Second, we treat the full range \(\rho>-1\), including the negative-exponent regime \(-1<\rho<0\), for which the damping operator \(v\mapsto\|v\|^\rho v\) fails to be Lipschitz at the origin; uniqueness is recovered through monotonicity and semicontinuity rather than Lipschitz estimates. Third, we obtain explicit decay rates--exponential for \(\rho=0\) and polynomial for \(\rho\neq0\)--without imposing geometric conditions on the underlying domain, thereby quantifying how the nonlocal nature of the damping affects the dissipation rate. To our knowledge, no comparable abstract result covering this combination and this range of exponents is available in the literature.

Building on the above studies, this paper investigates Problem \eqref{abstract}, an abstract model that combines viscoelastic damping with either nonlocal damping (\(\rho\neq0\)) or frictional damping (\(\rho=0\)). Using monotone operator theory and the Faedo--Galerkin method, we establish global well-posedness of the equation. Moreover, by means of integral inequalities, multiplier techniques, and Nakao's method, we prove that the energy decays exponentially when \(\rho=0\) and polynomially when \(\rho\neq0\), which reveals the influence of the nonlocal damping on the energy decay rate.

The paper is organized as follows. In Section \ref{sec2}, we introduce the functional setting and state the main results. Sections \ref{sec3} and \ref{sec4} are devoted to the global well-posedness and the energy decay estimates, respectively. In Section \ref{sec5}, we compare our results with the existing literature.

\section{Preliminaries and main results}\label{sec2}

Let \(V_0\) be a separable Hilbert space with scalar product \((\cdot,\cdot)\) and norm
\[
\|\varphi\|=\sqrt{(\varphi,\varphi)},\qquad \varphi\in V_0.
\]
For \(s\ge 0\), we define
\[
V_s:=D(A^{s/2}),\qquad
(\phi,\psi)_s:=(A^{s/2}\phi,A^{s/2}\psi),\qquad
\|\phi\|_s:=\|A^{s/2}\phi\|.
\]
In particular,
\[
V_1=D(A^{1/2}),\qquad V_2=D(A).
\]

We impose the following assumptions on the operator \(A\) and the relaxation function \(g\).

\begin{description}
\item[(AOP)] The operator \(A\) is a linear self-adjoint positive unbounded operator on \(V_0\). Let
\[
V_2:=\{\phi\in V_0:A\phi\in V_0\}.
\]
We assume that \(V_2\) is a Banach space with norm \(\|A(\cdot)\|\), that \(V_2\) is dense in \(V_1\), and that \(V_2\) is compactly embedded in \(V_0\). These spectral assumptions are standard in abstract viscoelasticity; see, e.g., \cite{MR0513814,MR2386941,MR2542886,MR4232931} and the references therein.

\item[(AG)] The function \(g:[0,\infty)\to[0,\infty)\) is non-increasing and locally absolutely continuous, with \(g(0)>0\) and \(g'(t)<0\) for almost every \(t\in[0,\infty)\). It satisfies
\begin{equation}\label{g}
g'(t)\leq -kg(t),\quad\text{a.e. }t\in[0,\infty),
\end{equation}
for some constant \(k>0\). Moreover,
\begin{equation}\label{mu_0}
\mu_0:=1-G_0>0,
\end{equation}
where
\begin{equation}\label{G0}
G_0:=\int_{0}^{\infty}g(t)dt.
\end{equation}
\end{description}

Equipped with the scalar product
\[
(\phi,\psi)_2:=(A\phi,A\psi),\qquad \phi,\psi\in V_2,
\]
the space \(V_2\) is a Hilbert space.

By Assumption (AOP) and standard spectral theory, there exists a countable orthonormal basis \(\{\omega_j\}\subset V_2\) of \(V_0\), consisting of eigenfunctions of \(A\):
\begin{equation}\label{basis}
\begin{split}
&A\omega_j=\lambda_j\omega_j,\\
&\|\omega_j\|=1,\quad j=1,2,\cdots,
\end{split}
\end{equation}
where the eigenvalues \(\{\lambda_j\}\) satisfy
\begin{align}\label{lam}
0<\lambda_1\leq\lambda_2\leq\cdots,\quad \lambda_j\to\infty\text{ as }j\to\infty.
\end{align}

For \(s>0\), define
\begin{equation}\label{Vs}
\begin{split}
&V_s:=\left\{\phi\in V_0:\sum_{j=1}^\infty\lambda_j^s(\phi,\omega_j)^2<\infty\right\},\\
&A^{\frac{s}{2}}:V_s\to V_0,\quad
A^{\frac{s}{2}}\phi=\sum_{j=1}^\infty\lambda_j^{\frac{s}{2}}(\phi,\omega_j)\omega_j,\quad \phi\in V_s.
\end{split}
\end{equation}
Then \(V_s\) is a Hilbert space with scalar product
\begin{equation}\label{inners}
(\phi,\psi)_s=\left(A^{\frac{s}{2}}\phi,A^{\frac{s}{2}}\psi\right),\qquad \phi,\psi\in V_s,
\end{equation}
and norm
\begin{equation}\label{innernorm}
\|\phi\|_s=\left\|A^{\frac{s}{2}}\phi\right\|=\left(\sum_{j=1}^\infty\lambda_j^s(\phi,\omega_j)^2\right)^{\frac{1}{2}},\qquad \phi\in V_s.
\end{equation}

For \(s>0\), \(V_{-s}\) denotes the dual space of \(V_s\), with duality product \((\cdot,\cdot)_{-s,s}\). The operator \(A^{\frac{s}{2}}\) extends to an isomorphism from \(V_0\) onto \(V_{-s}\). Moreover, \(V_{-s}\) is a Hilbert space with scalar product and norm given by \eqref{inners} and \eqref{innernorm} with \(s\) replaced by \(-s\). Equivalently, \(V_{-s}\) is the completion of \(V_0\) with respect to \(\|\cdot\|_{-s}\), i.e., \(V_{-s}=\overline{V_0}^{\|\cdot\|_{-s}}\).

For any \(s_1,s_2\in\mathbb{R}\), the embedding \(V_{s_1}\hookrightarrow V_{s_2}\) is continuous if \(s_1\geq s_2\), and compact if \(s_1>s_2\). Moreover,
\begin{align}\label{embed}
\left\|A^{\frac{s_2}{2}}\phi\right\|\leq\lambda_1^{\frac{s_2-s_1}{2}}\left\|A^{\frac{s_1}{2}}\phi\right\|,\qquad \phi\in V_{s_1}.
\end{align}

Throughout the paper, we set
\begin{equation}\label{Huah}
\mathcal{H}:=V_1\times V_0,\qquad \mathcal{H}_1:=V_2\times V_1.
\end{equation}
Both \(\mathcal{H}\) and \(\mathcal{H}_1\) are Hilbert spaces with scalar products
\[
((u_1;v_1),(u_2;v_2))_{\mathcal{H}}=(u_1,u_2)_1+(v_1,v_2),\qquad (u_i;v_i)\in\mathcal{H},
\]
and
\[
((u_1;v_1),(u_2;v_2))_{\mathcal{H}_1}=(u_1,u_2)_2+(v_1,v_2)_1,\qquad (u_i;v_i)\in\mathcal{H}_1,
\]
respectively. The corresponding norms are
\[
\|(u;v)\|_{\mathcal{H}}=\sqrt{\|u\|_1^2+\|v\|^2}=\sqrt{\left\|A^{\frac{1}{2}}u\right\|^2+\|v\|^2},\qquad (u;v)\in\mathcal{H},
\]
and
\[
\|(u;v)\|_{\mathcal{H}_1}=\sqrt{\|u\|_2^2+\|v\|_1^2}=\sqrt{\|Au\|^2+\left\|A^{\frac{1}{2}}v\right\|^2},\qquad (u;v)\in\mathcal{H}_1.
\]

\begin{definition}\label{defweak}
Assume that (AOP) and (AG) hold, and let \(\rho>-1\). A function
\(u\in C([0,\infty);V_1)\) with \(u_t\in C([0,\infty);V_0)\) satisfying
\((u(0);u_t(0))=(u_0;u_1)\) is called a \textit{global strong solution}
to Problem \eqref{abstract} on \([0,\infty)\) if:
\begin{description}
\item[(i)] \(u\in W^{1,\infty}(0,T;V_1)\) and \(u_t\in W^{1,\infty}(0,T;V_0)\)
for all \(0<T<\infty\);
\item[(ii)] Equation \eqref{abstract} holds in \(V_0\) for almost every
\(t\in[0,\infty)\).
\end{description}
\end{definition}

\begin{theorem}\label{thmwellposedness}
Let Assumptions (AOP) and (AG) hold and let \(\rho>-1\). For every initial data
\((u_0;u_1)\in\mathcal{H}_1\), Problem \eqref{abstract} admits a unique global solution
\(u\) on \([0,\infty)\). Moreover, this solution satisfies
\[
(u;u_t)\in L^{\infty}(0,\infty;\mathcal{H}_1),\qquad
u_{tt}\in L^{\infty}(0,\infty;V_0),
\]
and hence
\[
u\in C([0,\infty);V_1),\qquad u_t\in C([0,\infty);V_0).
\]
Consequently, \(u\) is a global solution in the sense of Definition~\ref{defweak}.
\end{theorem}

Let \(u=u(t)\) be the global solution to Problem \eqref{abstract} obtained in Theorem \ref{thmwellposedness}. The energy functional associated with Problem \eqref{abstract} is defined by
\begin{align}\label{energy}
E(t):=&\frac{1}{2}\left(\|u_t\|^2+\left(1-\int_{0}^{t}g(s)ds\right)\left\|A^{\frac{1}{2}}u\right\|^2\right)\notag\\
&+\frac{1}{2}\int_{0}^{t}g(t-s)\left\|A^{\frac{1}{2}}u(t)-A^{\frac{1}{2}}u(s)\right\|^2ds.
\end{align}

Define
\begin{equation}\label{dtG}
\widetilde{G}(t):=\int_0^t g(s)ds.
\end{equation}
By Assumption (AG), \(\widetilde{G}\in C([0,\infty))\), \(\widetilde{G}(0)=0\), and \(\widetilde{G}(t)\to G_0\) as \(t\to\infty\), where \(G_0\) is given in \eqref{G0}. Hence there exists \(t_0>0\) such that
\begin{equation}\label{existencet0}
\widetilde{G}(t_0)=\int_0^{t_0}g(s)ds=\frac{1}{2}G_0.
\end{equation}

\begin{theorem}\label{main}
Assume that (AOP) and (AG) hold, and let \(E(0)>0\).
\begin{enumerate}
\item For \(-1<\rho<0\),
\[
E(t)\leq C(\rho,\mu_0,\lambda_1,g(0),G_0,k)
\left(1+E(0)^{-\frac{\rho}{2\rho+2}}\right)^{\frac{2\rho+2}{|\rho|}}
t^{\frac{2\rho+2}{\rho}},\qquad t\geq t_0,
\]
where \(C(\rho,\mu_0,\lambda_1,g(0),G_0,k)>0\) is a constant depending only on the indicated parameters, and \(t_0>0\) is defined by
\[
\int_0^{t_0}g(s)ds=\frac12 G_0,
\]
with \(G_0=\int_0^\infty g(s)ds\) given in \eqref{G0}. Equivalently, since
\(-\frac{\rho}{2\rho+2}\in(0,\frac12)\) and
\(-\frac{2\rho+2}{\rho}=\frac{2\rho+2}{|\rho|}>0\), one may replace the
constant factor above by
\(\left(1+E(0)^{-1}\right)^{\frac{2\rho+2}{|\rho|}}\), at the cost of
enlarging the constant \(C\).

\item For \(\rho\geq0\), if the constant \(k\) in Assumption (AG) satisfies \(k\geq2g(0)\), then
\[
E(t)\leq
\begin{cases}
E(0)e^{-\kappa(t-1)^+}, & \rho=0,\\[1ex]
\left[\dfrac{\rho}{2\mu}(t-1)^++E(0)^{-\frac{\rho}{2}}\right]^{-\frac{2}{\rho}}, & \rho>0,
\end{cases}
\qquad t\geq0.
\]
Here
\[
\mu:=\left[2(E(0))^{\frac{\rho}{\rho+2}}+4\mu_1\right]^{1+\frac{\rho}{2}},\qquad
\kappa=-\log(1-1/\mu).
\]
Since \(\mu_1>0\) and \(\rho>0\), we have \(\mu>1\), so the constant
\(\kappa\) is well defined. Moreover,
\[
\mu_1:=\frac{32}{\mu_0\lambda_1}+\frac{4(2E(0))^\rho}{\mu_0\lambda_1}+\frac{3}{2}+\frac{2}{k}(E(0))^{\frac{\rho}{\rho+2}},
\]
with \(s^+:=\frac{1}{2}(s+|s|)\), and the positive constants \(\mu_0\) and \(\lambda_1\) given in \eqref{mu_0} and \eqref{lam}, respectively.
\end{enumerate}
\end{theorem}

\begin{remark}
The restriction \(\rho>-1\) is essential for the present approach. For \(\rho>-1\), the operator \(l(v)=\|v\|^\rho v\) is continuous, monotone, and semicontinuous on \(V_0\), and \(l(0)=0\). These properties are used both in the Faedo--Galerkin compactness argument and in the uniqueness proof. When \(\rho\le -1\), the operator \(l\) is no longer continuous and monotone at the origin; although the dissipation rate \(\|u_t\|^{\rho+2}\) remains nonnegative for \(\rho>-2\), the well-posedness theory requires a different approach. We therefore leave the case \(\rho\le -1\) for future work.
\end{remark}

\section{Well-posedness}\label{sec3}

In this section we prove the global well-posedness of Problem \eqref{abstract} by combining the Faedo--Galerkin method with the theory of monotone operators. We first establish the monotonicity and semicontinuity of the nonlocal damping operator.

\begin{lemma}\label{continuity}
Assume that \(\rho>-1\), and let
\[
l(v):=\|v\|^{\rho}v,\qquad v\in V_0.
\]
Then \(l:V_0\to V_0\) is monotone and semicontinuous. More precisely:
\begin{enumerate}
\item \((l(u)-l(v),u-v)\geq0\) for all \(u,v\in V_0\);
\item for all \(u,v,w\in V_0\), the map \(\lambda\mapsto(l(u+\lambda v),w)\) is continuous from \(\mathbb{R}\) to \(\mathbb{R}\).
\end{enumerate}
\end{lemma}

\begin{proof}
If \(u=0\), then
\[
(l(u)-l(v),u-v)=(-\|v\|^{\rho}v,-v)=\|v\|^{\rho+2}\ge0;
\]
if \(v=0\), then
\[
(l(u)-l(v),u-v)=(\|u\|^{\rho}u,u)=\|u\|^{\rho+2}\ge0.
\]
Hence assume \(u\neq0\) and \(v\neq0\). For \(u,v\in V_0\), we have
\begin{align*}
(l(u)-l(v),u-v)
&=(\|u\|^{\rho}u-\|v\|^{\rho}v,u-v)\\
&=\|u\|^{\rho}(\|u\|^2-(u,v))+\|v\|^{\rho}(\|v\|^2-(u,v))\\
&\geq\|u\|^{\rho+2}-\|u\|^{\rho+1}\|v\|+\|v\|^{\rho+2}-\|v\|^{\rho+1}\|u\|\\
&=(\|u\|-\|v\|)(\|u\|^{\rho+1}-\|v\|^{\rho+1}).
\end{align*}
Since \(\rho+1>0\), the function \(t\mapsto t^{\rho+1}\) is nondecreasing on \([0,\infty)\), and hence the last expression is nonnegative. This proves (i).

To prove (ii), fix \(u,v,w\in V_0\) and set
\[
\phi(\lambda):=(l(u+\lambda v),w)=\|u+\lambda v\|^{\rho}((u,w)+\lambda(v,w)),\qquad \lambda\in\mathbb{R}.
\]
If \(u=0\) and \(v=0\), then \(\phi\equiv0\). If \(u=0\) and \(v\neq0\), then
\[
\phi(\lambda)=|\lambda|^{\rho}\lambda\|v\|^{\rho}(v,w),
\]
which is continuous at \(\lambda=0\) because \(\rho>-1\) implies \(\rho+1>0\), and continuous elsewhere by the continuity of \(\lambda\mapsto|\lambda|^\rho\lambda\) on \(\mathbb{R}\setminus\{0\}\). If \(u\neq0\) and \(v=0\), then \(\phi\) is constant.

Finally, if \(u\neq0\) and \(v\neq0\), we distinguish two cases.

\medskip
\noindent\emph{Case 1: there exists \(\lambda_0\in\mathbb{R}\) such that \(u+\lambda_0 v=0\).}
Such a \(\lambda_0\) is unique. On \(\mathbb{R}\setminus\{\lambda_0\}\), the continuity of \(\phi\) follows from the continuity of \(\lambda\mapsto\|u+\lambda v\|\) and of \(t\mapsto t^\rho\) on \((0,\infty)\). At \(\lambda=\lambda_0\), we have \(u+\lambda_0 v=0\), hence \((u,w)+\lambda_0(v,w)=0\), and
\[
\phi(\lambda)=\|u+\lambda v\|^{\rho}((\lambda-\lambda_0)(v,w))
=|\lambda-\lambda_0|^{\rho}(\lambda-\lambda_0)\|v\|^{\rho}(v,w)\to0
\]
as \(\lambda\to\lambda_0\), since \(\rho+1>0\). Thus \(\phi\) is continuous at \(\lambda_0\) as well, hence on all of \(\mathbb{R}\).

\medskip
\noindent\emph{Case 2: no such \(\lambda_0\) exists.}
Then \(u+\lambda v\neq0\) for every \(\lambda\in\mathbb{R}\), so \(\|u+\lambda v\|>0\) for every \(\lambda\). The continuity of \(\phi\) on \(\mathbb{R}\) follows from the continuity of \(\lambda\mapsto\|u+\lambda v\|\) and of \(t\mapsto t^\rho\) on \((0,\infty)\).

Thus \(\phi\) is continuous in all cases. This proves (ii).
\end{proof}

\begin{proof}[Proof of Theorem \ref{thmwellposedness}]
Let \(\{\omega_k\}_{k=1}^\infty\) be the orthonormal basis of \(V_0\) given in \eqref{basis}. For each \(n\in\mathbb{N}\), we consider the approximate problem
\begin{align}
\left(u_{tt}^n(t),\omega_k\right)&+\left(Au^n(t),\omega_k\right)\notag\\
&-\left(\int_{0}^{t}g(t-s)Au^n(s)ds,\omega_k\right)
+\|u_t^n(t)\|^{\rho}\left(u_t^n(t),\omega_k\right)=0,\quad t>0,\label{weak}
\end{align}
for \(k=1,\dots,n\), where
\begin{equation}\label{bc1}
u^n(t)=\sum_{j=1}^n c_{nj}(t)\omega_j,\qquad u^n(0)=u^{n}_0,\qquad u^n_t(0)=u^{n}_1,
\end{equation}
and
\begin{align}
\left(u^{n}_0;u^{n}_1\right)\rightarrow(u_0;u_1)\quad\text{strongly in }\mathcal{H}_1\text{ as }n\to\infty.\label{defineun}
\end{align}

Substituting \eqref{bc1} into \eqref{weak}, we obtain the system of ordinary differential equations
\begin{align}
c''_{nk}&+\sum_{j=1}^{n}(A^{\frac12}\omega_j,A^{\frac12}\omega_k)c_{nj}\notag\\
&-\sum_{j=1}^{n}\left(A^{\frac12}\omega_j,A^{\frac12}\omega_k\right)
\int_{0}^{t}g(t-s)c_{nj}(s)ds
+F_k\left(c'_{n1},\dots,c'_{nn}\right)=0,\label{withouttime}
\end{align}
where
\[
F_k(c'_{n1},\dots,c'_{nn})
:=
\begin{cases}
\left(\sum_{j=1}^n(c'_{nj})^2\right)^{\rho/2}c'_{nk}, & \sum_{j=1}^n(c'_{nj})^2\ne0,\\[1ex]
0, & \sum_{j=1}^n(c'_{nj})^2=0.
\end{cases}
\]
Since \(\rho>-1\), the map \((c'_{n1},\dots,c'_{nn})\mapsto F_k(c'_{n1},\dots,c'_{nn})\) is continuous: indeed, with \(r=(\sum_{j=1}^n(c'_{nj})^2)^{1/2}\), we have \(|F_k|\leq r^{\rho+1}\to0\) as \(r\to0\), because \(\rho+1>0\). Moreover, the memory term is linear and continuous in \(c_{nj}\). Hence the right-hand side of \eqref{withouttime} is continuous in \((c_{nj},c'_{nj})\), and by the Peano existence theorem there exists a local solution \(c_{nk}\in C^2([0,T_n])\), \(k=1,\dots,n\), for some \(T_n>0\).

We now derive uniform a priori estimates on \([0,T_n]\). Multiplying \eqref{weak} by \(c'_{nk}(t)\) and summing over \(k=1,\dots,n\), we obtain
\begin{align}
\frac12\frac{d}{dt}&\left(\left\|u_t^n(t)\right\|^2+\left\|A^{\frac12}u^n(t)\right\|^2\right)\notag\\
&-\left(\int_{0}^{t}g(t-s)A^{\frac12}u^n(s)ds,A^{\frac12}u_t^n(t)\right)
+\|u_t^n\|^{\rho+2}=0.\label{oneorder}
\end{align}
We use the standard identity for the memory term,
\begin{align}
&-\left(\int_{0}^{t}g(t-s)A^{\frac12}u^n(s)ds,A^{\frac12}u_t^n(t)\right)\notag\\
=&\frac12\frac{d}{dt}\int_{0}^{t}g(t-s)\left\|A^{\frac12}u^n(t)-A^{\frac12}u^n(s)\right\|^2ds
-\frac12\int_{0}^{t}g'(t-s)\left\|A^{\frac12}u^n(t)-A^{\frac12}u^n(s)\right\|^2ds\notag\\
&+\frac12g(t)\left\|A^{\frac12}u^n(t)\right\|^2
-\frac12\frac{d}{dt}\left(\int_{0}^{t}g(s)ds\left\|A^{\frac12}u^n(t)\right\|^2\right).\label{import}
\end{align}
Substituting \eqref{import} into \eqref{oneorder} and integrating from \(0\) to \(t\), we obtain
\begin{align}
E_n(t)+\int_{0}^{t}\left\|u_t^n(s)\right\|^{\rho+2}ds=&E_n(0)+\frac12\int_{0}^{t}\int_{0}^{s}g'(s-\tau)\left\|A^{\frac12}u^n(s)-A^{\frac12}u^n(\tau)\right\|^2d\tau ds\notag\\
&-\frac12\int_{0}^{t}g(s)\left\|A^{\frac12}u^n(s)\right\|^2ds,\label{energyidentity}
\end{align}
where
\begin{align}
E_n(t):=&\frac12\left\|u_t^n(t)\right\|^2
+\frac12\left(1-\int_{0}^{t}g(s)ds\right)\left\|A^{\frac12}u^n(t)\right\|^2\notag\\
&+\frac12\int_{0}^{t}g(t-s)\left\|A^{\frac12}u^n(t)-A^{\frac12}u^n(s)\right\|^2ds.\label{En}
\end{align}
Since
\[
E_n(0)=\frac12\left(\|u_1^n\|^2+\left\|A^{\frac12}u_0^n\right\|^2\right)
\le C\left(\|(u_0^n;u_1^n)\|_{\mathcal{H}_1}\right),
\]
and \((u_0^n;u_1^n)\) is uniformly bounded in \(\mathcal{H}_1\) by \eqref{defineun}, the term \(E_n(0)\) is uniformly bounded. Since \(g'\leq0\) and \(g\geq0\), the last two terms on the right-hand side of \eqref{energyidentity} are nonpositive. Hence, by \eqref{mu_0},
\begin{align}\label{weakestimateone}
\left\|u_t^n(t)\right\|^2&+\left\|A^{\frac12}u^n(t)\right\|^2\notag\\
&+\int_0^t g(t-s)\left\|A^{\frac12}u^n(t)-A^{\frac12}u^n(s)\right\|^2ds
\leq C\left(\|(u_0;u_1)\|_{\mathcal{H}_1}\right),
\end{align}
for all \(t\in[0,T_n]\). In particular, by \eqref{defineun},
\begin{equation}\label{weakestimateoneb}
\left\|u_t^n(t)\right\|^2+\left\|A^{\frac12}u^n(t)\right\|^2\leq C\left(\|(u_0;u_1)\|_{\mathcal{H}_1}\right),\qquad t\in[0,T_n].
\end{equation}

Next, multiplying \eqref{weak} by \(\lambda_kc'_{nk}(t)\) and summing over \(k=1,\dots,n\), we obtain
\begin{align*}
\frac12\frac{d}{dt}&\left(\left\|A^{\frac12}u_t^n(t)\right\|^2+\left\|Au^n(t)\right\|^2\right)
\\
&-\left(\int_{0}^{t}g(t-s)Au^n(s)ds,Au_t^n(t)\right)
+\|u_t^n\|^{\rho}\left\|A^{\frac12}u^n_t\right\|^2=0.
\end{align*}
Using the same identity for the memory term with \(A^{1/2}u^n\) replaced by \(Au^n\), and noting that \(\|u_t^n\|^{\rho}\|A^{1/2}u_t^n\|^2\geq0\), we deduce
\begin{align}
\left\|A^{\frac12}u_t^n(t)\right\|^2+\left\|Au^n(t)\right\|^2\leq C\left(\|(u_0;u_1)\|_{\mathcal{H}_1}\right),\qquad t\in[0,T_n].\label{weaktwo}
\end{align}

Finally, multiplying \eqref{weak} by \(c''_{nk}(t)\) and summing over \(k=1,\dots,n\), we get
\[
\|u_{tt}^n\|^2=\sum_{i=1}^{3}Y_i,
\]
where
\[
Y_1:=-\left(Au^n(t),u_{tt}^n\right),\quad
Y_2:=\left(\int_{0}^{t}g(t-s)Au^n(s)ds,u_{tt}^n\right),\quad
Y_3:=-\|u_t^n\|^{\rho}\left(u_t^n,u_{tt}^n\right).
\]
By Cauchy's inequality, Young's inequality, \eqref{mu_0}, \eqref{G0}, \eqref{weakestimateoneb} and \eqref{weaktwo}, we obtain
\[
|Y_1|\leq\frac16\left\|u_{tt}^n\right\|^2+C\left(\|(u_0;u_1)\|_{\mathcal{H}_1}\right),
\]
\[
|Y_2|\leq\frac16\left\|u_{tt}^n\right\|^2+C\left(\|(u_0;u_1)\|_{\mathcal{H}_1}\right),
\]
\[
|Y_3|\leq\frac16\left\|u_{tt}^n\right\|^2+C\left(\|(u_0;u_1)\|_{\mathcal{H}_1}\right).
\]
Indeed, for \(Y_2\) we used
\[
\left\|\int_0^t g(t-s)Au^n(s)ds\right\|
\leq G_0\|Au^n\|_{L^\infty(0,\infty;V_0)}\leq C,
\]
and for \(Y_3\) we used \(\|u_t^n\|\leq C\) together with the fact that,
since \(\rho+1>0\), the function \(s\mapsto s^{\rho+1}\) is continuous on
\([0,\infty)\) and hence bounded on \([0,C]\), which gives
\(\|u_t^n\|^{\rho+1}\leq C\). Hence
\begin{align}
\|u_{tt}^n\|^2\leq C(\|(u_0;u_1)\|_{\mathcal{H}_1}),\qquad t\in[0,T_n].\label{utt}
\end{align}
Combining \eqref{defineun}, \eqref{weakestimateoneb}, \eqref{weaktwo} and \eqref{utt}, we arrive at
\begin{equation}\label{bc8}
\|u_{t}^n(t)\|^2+\|u_{tt}^n(t)\|^2+\left\|A^{\frac12}u_t^n(t)\right\|^2+\left\|Au^n(t)\right\|^2\leq C\left(\|(u_0;u_1)\|_{\mathcal{H}_1}\right),\quad t\in[0,T_n].
\end{equation}
Since the bound on the right-hand side of \eqref{bc8} is independent of \(T_n\), the local solution can be extended to all \(t\geq0\), and hence \(T_n=\infty\). Consequently, the estimates \eqref{weakestimateoneb}, \eqref{weaktwo}, \eqref{utt} and \eqref{bc8} hold for all \(t\in[0,\infty)\).

From \eqref{bc8}, \eqref{G0} and \(\rho>-1\), we further obtain, for all \(t\in[0,\infty)\),
\[
\left\|\|u_t^n(t)\|^{\rho}u_t^n(t)\right\|=\|u_t^n(t)\|^{\rho+1}\leq C\left(\|(u_0;u_1)\|_{\mathcal{H}_1}\right),
\]
since \(\|u_t^n(t)\|\leq C\) by \eqref{bc8} and the function \(s\mapsto s^{\rho+1}\) is continuous on \([0,\infty)\) with \(\rho+1>0\). Moreover,
\begin{align*}
 \left\|\int_{0}^{t}g(t-s)A^{\frac12}u^n(s)ds\right\|_1
&=\left\|A^{\frac12}\int_{0}^{t}g(t-s)A^{\frac12}u^n(s)ds\right\|
\leq\int_{0}^{t}g(t-s)\|Au^n(s)\|ds
\\
&\leq G_0\|Au^n\|_{L^\infty(0,\infty;V_0)}
\leq C\left(\|(u_0;u_1)\|_{\mathcal{H}_1}\right).
\end{align*}
Therefore, there exist a subsequence of \(\{u^n\}\), still denoted by \(\{u^n\}\), a function \(u\in L^{\infty}(0,\infty;V_2)\) with \(u_t\in L^{\infty}(0,\infty;V_1)\), and a function \(\chi\in L^\infty(0,\infty;V_0)\) such that, as \(n\to\infty\),
\begin{equation}\label{strongsolu}
\left\{
\begin{split}
&u^n\rightharpoonup u\quad \text{weakly star in }L^{\infty}(0,\infty;V_2),\\
&u_t^n\rightharpoonup u_t\quad \text{weakly star in }L^{\infty}(0,\infty;V_1),\\
&u_{tt}^n\rightharpoonup u_{tt}\quad \text{weakly star in }L^{\infty}(0,\infty;V_{0}),\\
&\int_{0}^{t}g(t-s)A^{\frac12}u^n(s)ds\rightharpoonup\int_{0}^{t}g(t-s)A^{\frac12}u(s)ds\quad \\
&\qquad\qquad\qquad\qquad\qquad\text{weakly star in }L^{\infty}(0,\infty;V_1),\\
&\|u_t^n\|^{\rho}u_t^n\rightharpoonup\chi\quad \text{weakly star in }L^{\infty}(0,\infty;V_0).
\end{split}
\right.
\end{equation}

Since \(V_2\hookrightarrow V_1\hookrightarrow V_0\) compactly, the Aubin--Lions--Simon theorem \cite{MR0916688} together with \eqref{bc8} and \eqref{strongsolu} yields, after passing to a subsequence,
\begin{equation}\label{bc5}
\begin{split}
&u_{t}^n\rightarrow u_{t}\quad \text{strongly in }C([0,T];V_{0})\text{ for every }T>0,\\
&u^n\rightarrow u\quad \text{strongly in }C([0,T];V_{1})\text{ for every }T>0.
\end{split}
\end{equation}
In particular, \(u^n(0)\to u(0)\) strongly in \(V_1\) and \(u^n_t(0)\to u_t(0)\) strongly in \(V_{0}\). By \eqref{defineun},
\begin{equation}\label{initialdy}
u(0)=u_0,\qquad u_t(0)=u_1.
\end{equation}

We next identify the weak limit \(\chi\). Since \(l:V_0\to V_0\) is monotone and semicontinuous, we argue as in the classical monotone operator method. For any \(\varphi\in L^2(0,t;V_0)\), the monotonicity of \(l\) gives
\[
\int_{0}^{t}\left(\|u_{\tau}^n\|^{\rho}u_{\tau}^n-\|\varphi\|^{\rho}\varphi,u_{\tau}^n-\varphi\right)d\tau\geq0.
\]
Taking \(\liminf\) as \(n\to\infty\), and using \eqref{strongsolu} and \eqref{bc5}, we obtain
\[
\int_{0}^{t}\left(\chi-\|\varphi\|^{\rho}\varphi,u_{\tau}-\varphi\right)d\tau\geq0.
\]
Here we used that \(\|u_{\tau}^n\|^{\rho}u_{\tau}^n\rightharpoonup\chi\) weakly in \(L^2(0,t;V_0)\) and \(u_{\tau}^n\to u_{\tau}\) strongly in \(L^2(0,t;V_0)\), so that the product converges. Choosing \(\varphi=u_\tau-\lambda\omega\) with \(\lambda\geq0\) and \(\omega\in L^2(0,t;V_0)\), and then letting \(\lambda\to0^+\), we get
\[
\int_{0}^{t}\left(\chi-\|u_\tau\|^{\rho}u_\tau,\omega\right)d\tau\geq0.
\]
Repeating the argument with \(\lambda\leq0\) yields the reverse inequality. Hence
\begin{equation}\label{chiey}
\chi=\|u_t\|^{\rho}u_t.
\end{equation}

Now, letting \(n\to\infty\) in \eqref{weak}, and using \eqref{strongsolu} and \eqref{chiey}, we obtain
\[
(u_{tt}(t),\omega_k)+(Au(t),\omega_k)-\left(\int_{0}^{t}g(t-s)Au(s)ds,\omega_k\right)+\|u_t\|^{\rho}(u_t,\omega_k)=0
\]
for \(t\in(0,\infty)\) and \(k=1,2,\cdots\). Since \(\{\omega_k\}_{k=1}^\infty\) is a basis of \(V_0\), it follows that
\[
u_{tt}(t)+Au(t)-\int_{0}^{t}g(t-s)Au(s)ds+\|u_t\|^{\rho}u_t=0
\]
holds in \(V_0\) for almost every \(t\in(0,\infty)\). Together with \eqref{initialdy}, this proves the existence part of the theorem.

Moreover, the above estimates show that \((u;u_t)\in L^{\infty}(0,\infty;\mathcal{H}_1)\) and \(u_{tt}\in L^{\infty}(0,\infty;V_0)\). Since \(u\in L^\infty(0,\infty;V_2)\), \(u_t\in L^\infty(0,\infty;V_1)\), we have \(u\in W^{1,\infty}(0,T;V_1)\subset C([0,T];V_1)\) for every \(T>0\). Similarly, since \(u_t\in L^\infty(0,\infty;V_1)\) and \(u_{tt}\in L^\infty(0,\infty;V_0)\), we have \(u_t\in W^{1,\infty}(0,T;V_0)\subset C([0,T];V_0)\) for every \(T>0\). Thus \(u\in C([0,\infty);V_1)\) and \(u_t\in C([0,\infty);V_0)\), which matches Definition~\ref{defweak}.

It remains to prove uniqueness. Let \(u\) and \(v\) be two global solutions of \eqref{abstract} with initial data \(z_0=(u_0;u_1)\in\mathcal{H}_1\) and \(\widetilde{z}_0=(v_0;v_1)\in\mathcal{H}_1\), respectively, and set \(\omega:=u-v\). Then
\begin{equation}\label{unique}
\omega_{tt}+A\omega-\int_{0}^{t}g(t-s)A\omega(s)ds+\|u_t\|^{\rho}u_t-\|v_t\|^{\rho}v_t=0.
\end{equation}
Taking the multiplier \(\omega_t\) in \eqref{unique}, we obtain
\[
\frac12\frac{d}{dt}\left(\|\omega_t\|^2+\|A^{1/2}\omega\|^2\right)
-\left(\int_0^t g(t-s)A\omega(s)ds,\omega_t\right)
+\left(\|u_t\|^\rho u_t-\|v_t\|^\rho v_t,\omega_t\right)=0.
\]
Define
\begin{align*}
 E_\omega(t):=&\frac12\left(\|\omega_t(t)\|^2+\left(1-\int_0^t g(s)ds\right)\|A^{1/2}\omega(t)\|^2\right)\\
&+\frac12\int_0^t g(t-s)\|A^{1/2}\omega(t)-A^{1/2}\omega(s)\|^2ds.
\end{align*}
Using the standard identity for the memory term, we get
\begin{align*}
E_\omega'(t)&+\left(\|u_t\|^\rho u_t-\|v_t\|^\rho v_t,\omega_t\right)\\
&+\frac12g(t)\|A^{1/2}\omega(t)\|^2
-\frac12\int_0^t g'(t-s)\|A^{1/2}\omega(t)-A^{1/2}\omega(s)\|^2ds=0.
\end{align*}
Equivalently,
\begin{align*}
 E_\omega'(t)
=
&-\left(\|u_t\|^\rho u_t-\|v_t\|^\rho v_t,\omega_t\right)\\
&-\frac12g(t)\|A^{1/2}\omega(t)\|^2
+\frac12\int_0^t g'(t-s)\|A^{1/2}\omega(t)-A^{1/2}\omega(s)\|^2ds.
\end{align*}
By Lemma \ref{continuity}, the map \(l(v)=\|v\|^\rho v\) is monotone for \(\rho>-1\). Hence
\[
\left(\|u_t\|^\rho u_t-\|v_t\|^\rho v_t,\omega_t\right)
=(l(u_t)-l(v_t),u_t-v_t)\ge0.
\]
Since \(g\ge0\) and \(g'\le0\), we deduce that \(E_\omega'(t)\le0\). But \(E_\omega(0)=0\) and \(E_\omega(t)\ge0\), so \(E_\omega(t)=0\) for all \(t\ge0\). Therefore \(\omega\equiv0\), and uniqueness follows for every \(\rho>-1\).
\end{proof}

\section{Energy decay estimates}\label{sec4}

This section is devoted to the proof of Theorem \ref{main}. We distinguish two cases: \(-1<\rho<0\) and \(\rho\ge0\). Throughout this section, \(u=u(t)\), \(t\in[0,\infty)\), denotes the global solution to Problem \eqref{abstract} obtained in Theorem \ref{thmwellposedness}, and \(C\) denotes a generic positive constant which may vary from line to line. When a constant depends on specific parameters, this dependence will be indicated explicitly.

\noindent\textbf{The case $\bm{-1<\rho<0}$.} The main difficulty in the case \(-1<\rho<0\) is that the nonlocal damping operator is not Lipschitz at zero. Nevertheless, it remains monotone, which is enough to prove uniqueness. The memory term provides additional dissipation through
\[
\int_0^t g(t-s)\left\|A^{1/2}u(t)-A^{1/2}u(s)\right\|^2ds,
\]
which compensates for the singular behavior of the nonlocal damping and allows us to construct a suitable Lyapunov functional and apply Nakao's iteration method.

For \(\delta\in(0,1)\), define
\begin{align}
&K_\delta(s):=-g'(s)+\delta g(s),\label{K}\\
&M(\delta):=\int_{0}^{\infty}\frac{g^2(s)}{K_\delta(s)}ds,\label{M}\\
&G(t):=\int_{t}^{\infty}g(s)ds.\label{G}
\end{align}
By Assumption (AG), \(K_\delta(s)>0\) for almost every \(s\ge0\), and \(M(\delta)<\infty\) for every \(\delta>0\).

\begin{lemma}\label{lemI1}
Assume that (AOP) and (AG) hold, and define
\begin{equation}\label{I1t}
I_1(t):=-\left(u_t(t),\int_{0}^{t}g(t-s)(u(t)-u(s))ds\right).
\end{equation}
For any \(\varepsilon_1>0\), there exists a constant \(\hat C(\varepsilon_1)>0\), depending only on \(\varepsilon_1\) and \(\lambda_1\), such that, for any \(t\geq t_0\),
\begin{align}\label{I1esti}
\frac{d}{dt}I_1(t)
&\leq-\frac{G_0}{4}\|u_t(t)\|^2+\varepsilon_1\left\|A^{\frac12}u(t)\right\|^2+C\|u_t(t)\|^{2\rho+2}\notag\\
&\quad+\hat C(\varepsilon_1)M(\delta)\int_{0}^{t}K_\delta(t-s)\left\|A^{\frac12}u(t)-A^{\frac12}u(s)\right\|^2ds\notag\\
&\quad-\frac{g(0)}{\lambda_1G_0}\int_{0}^{t}g'(t-s)\left\|A^{\frac12}u(t)-A^{\frac12}u(s)\right\|^2ds,
\end{align}
where \(t_0\), \(G_0\), \(\lambda_1\), \(K_\delta\), and \(M(\delta)\) are given in \eqref{existencet0}, \eqref{G0}, \eqref{basis}, \eqref{K}, and \eqref{M}, respectively, and \(C>0\) is a constant depending only on \(\rho\).
\end{lemma}

\begin{proof}
Using \eqref{abstract} and the definition of \(I_1\), together with Young's inequality and \eqref{embed} with \(s_1=1\) and \(s_2=0\), we obtain
\begin{align}
\frac{d}{dt}I_1(t)
&=-\int_{0}^{t}g(s)ds\|u_t(t)\|^2
-\left(u_{tt},\int_{0}^{t}g(t-s)(u(t)-u(s))ds\right)\notag\\
&\quad-\left(u_t(t),\int_{0}^{t}g'(t-s)(u(t)-u(s))ds\right)\notag\\
&=-\int_{0}^{t}g(s)ds\|u_t(t)\|^2
+\left(A u(t),\int_{0}^{t}g(t-s)(u(t)-u(s))ds\right)\notag\\
&\quad-\left(\int_{0}^{t}g(t-s)A u(s)ds,\int_{0}^{t}g(t-s)(u(t)-u(s))ds\right)\notag\\
&\quad+\|u_t(t)\|^{\rho}\left(u_t(t),\int_{0}^{t}g(t-s)(u(t)-u(s))ds\right)\notag\\
&\quad-\left(u_t(t),\int_{0}^{t}g'(t-s)(u(t)-u(s))ds\right)\notag\\
&=-\int_{0}^{t}g(s)ds\|u_t(t)\|^2\notag\\
&\quad+\left(A^{\frac12}u(t),\int_{0}^{t}g(t-s)(A^{\frac12}u(t)-A^{\frac12}u(s))ds\right)\notag\\
&\quad-\left(\int_{0}^{t}g(t-s)A^{\frac12}u(s)ds,\int_{0}^{t}g(t-s)(A^{\frac12}u(t)-A^{\frac12}u(s))ds\right)\notag\\
&\quad+\|u_t(t)\|^{\rho}\left(u_t(t),\int_{0}^{t}g(t-s)(u(t)-u(s))ds\right)\notag\\
&\quad-\left(u_t(t),\int_{0}^{t}g'(t-s)(u(t)-u(s))ds\right).\notag
\end{align}

Using the decomposition
\[
A^{\frac12}u(t)
=\left(1-\int_{0}^{t}g(s)ds\right)A^{\frac12}u(t)
+\int_{0}^{t}g(t-s)A^{\frac12}u(t)ds,
\]
we combine the second and third terms on the right-hand side as
\begin{align}
&\left(A^{\frac12}u(t),\int_{0}^{t}g(t-s)(A^{\frac12}u(t)-A^{\frac12}u(s))ds\right)\notag\\
&\quad-\left(\int_{0}^{t}g(t-s)A^{\frac12}u(s)ds,\int_{0}^{t}g(t-s)(A^{\frac12}u(t)-A^{\frac12}u(s))ds\right)\notag\\
&=\left(\left(1-\int_{0}^{t}g(s)ds\right)A^{\frac12}u(t),
\int_{0}^{t}g(t-s)(A^{\frac12}u(t)-A^{\frac12}u(s))ds\right)\notag\\
&\quad+\left\|\int_{0}^{t}g(t-s)(A^{\frac12}u(t)-A^{\frac12}u(s))ds\right\|^2.\notag
\end{align}
Hence
\begin{align}
\frac{d}{dt}I_1(t)
&=-\int_{0}^{t}g(s)ds\|u_t(t)\|^2\notag\\
&\quad+\left(\left(1-\int_{0}^{t}g(s)ds\right)A^{\frac12}u(t),
\int_{0}^{t}g(t-s)(A^{\frac12}u(t)-A^{\frac12}u(s))ds\right)\notag\\
&\quad+\left\|\int_{0}^{t}g(t-s)(A^{\frac12}u(t)-A^{\frac12}u(s))ds\right\|^2\notag\\
&\quad+\|u_t(t)\|^{\rho}\left(u_t(t),\int_{0}^{t}g(t-s)(u(t)-u(s))ds\right)\notag\\
&\quad-\left(u_t(t),\int_{0}^{t}g'(t-s)(u(t)-u(s))ds\right)\notag\\
&\leq-\int_{0}^{t}g(s)ds\|u_t(t)\|^2\notag\\
&\quad+\varepsilon_1\|A^{\frac12}u(t)\|^2
+\widetilde C(\varepsilon_1)\left\|\int_{0}^{t}g(t-s)(A^{\frac12}u(t)-A^{\frac12}u(s))ds\right\|^2\notag\\
&\quad+\left\|\int_{0}^{t}g(t-s)(A^{\frac12}u(t)-A^{\frac12}u(s))ds\right\|^2\notag\\
&\quad+\frac12\|u_t(t)\|^{2\rho+2}
+\frac12\left\|\int_{0}^{t}g(t-s)(u(t)-u(s))ds\right\|^2\notag\\
&\quad+\frac{G_0}{4}\|u_t(t)\|^2
+\frac{1}{G_0}\left\|\int_{0}^{t}-g'(t-s)(u(t)-u(s))ds\right\|^2.\notag
\end{align}
By the embedding \(\|u(t)-u(s)\|\le \lambda_1^{-1/2}\|A^{1/2}u(t)-A^{1/2}u(s)\|\), we have
\begin{align}
&\frac12\left\|\int_{0}^{t}g(t-s)(u(t)-u(s))ds\right\|^2
\leq \frac{1}{2\lambda_1}\left\|\int_{0}^{t}g(t-s)(A^{\frac12}u(t)-A^{\frac12}u(s))ds\right\|^2,\notag\\
&\frac{1}{G_0}\left\|\int_{0}^{t}-g'(t-s)(u(t)-u(s))ds\right\|^2
\leq \frac{1}{\lambda_1G_0}\left\|\int_{0}^{t}-g'(t-s)(A^{\frac12}u(t)-A^{\frac12}u(s))ds\right\|^2.\notag
\end{align}
Combining the above estimates, we arrive at
\begin{align}
\frac{d}{dt}I_1(t)
&\leq\left(-\int_{0}^{t}g(s)ds+\frac{G_0}{4}\right)\|u_t(t)\|^2
+\varepsilon_1\|A^{\frac12}u(t)\|^2
+\frac12\|u_t(t)\|^{2\rho+2}\notag\\
&\quad+\hat C(\varepsilon_1)\left\|\int_{0}^{t}g(t-s)(A^{\frac12}u(t)-A^{\frac12}u(s))ds\right\|^2\notag\\
&\quad+\frac{1}{\lambda_1G_0}\left\|\int_{0}^{t}-g'(t-s)(A^{\frac12}u(t)-A^{\frac12}u(s))ds\right\|^2,\notag
\end{align}
where
\[
\hat C(\varepsilon_1)=\widetilde C(\varepsilon_1)+1+\frac{1}{2\lambda_1}.
\]
Here we used Young's inequality in the form
\[
\|u_t\|^{\rho+1}\left\|\int_0^t g(t-s)(u(t)-u(s))ds\right\|
\le \frac12\|u_t\|^{2\rho+2}
+\frac12\left\|\int_0^t g(t-s)(u(t)-u(s))ds\right\|^2,
\]
together with the embedding \(\|u(t)-u(s)\|\le \lambda_1^{-1/2}\|A^{1/2}u(t)-A^{1/2}u(s)\|\).

Since \(t\geq t_0\), \eqref{existencet0} gives
\[
-\int_{0}^{t}g(s)ds+\frac{G_0}{4}
\leq-\int_{0}^{t_0}g(s)ds+\frac{G_0}{4}
=-\frac{G_0}{4}.
\]
Hence the coefficient of \(\|u_t(t)\|^2\) is bounded by \(-\frac{G_0}{4}\).

Next, by the definition of \(K_\delta\) and \(M(\delta)\), together with H\"older's inequality,
\begin{align}
&\left\|\int_{0}^{t}g(t-s)\left(A^{\frac12}u(t)-A^{\frac12}u(s)\right)ds\right\|^2\nonumber\\
=&\left\|\int_{0}^{t}\frac{g(t-s)}{\sqrt{K_\delta(t-s)}}\sqrt{K_\delta(t-s)}\left(A^{\frac12}u(t)-A^{\frac12}u(s)\right)ds\right\|^2\nonumber\\
\leq&\int_{0}^{t}\frac{g^2(t-s)}{K_\delta(t-s)}ds\int_0^tK_\delta(t-s)\left\|A^{\frac12}u(t)-A^{\frac12}u(s)\right\|^2ds\nonumber\\
=&M(\delta)\int_0^tK_\delta(t-s)\left\|A^{\frac12}u(t)-A^{\frac12}u(s)\right\|^2ds.\label{Kd}
\end{align}
Also, since \(-g'\ge0\) and \(\int_0^t -g'(t-s)ds\le g(0)\),
\begin{align*}
&\left\|\int_{0}^{t}-g'(t-s)\left(A^{\frac12}u(t)-A^{\frac12}u(s)\right)ds\right\|^2\\
=&\left\|\int_{0}^{t}\sqrt{-g'(t-s)}\sqrt{-g'(t-s)}\left(A^{\frac12}u(t)-A^{\frac12}u(s)\right)ds\right\|^2\\
&\leq\int_0^t-g'(t-s)ds\int_0^t-g'(t-s)\left\|A^{\frac12}u(t)-A^{\frac12}u(s)\right\|^2ds\\
&\leq g(0)\int_{0}^{t}-g'(t-s)\left\|A^{\frac12}u(t)-A^{\frac12}u(s)\right\|^2ds.
\end{align*}
Inserting these two estimates into the previous inequality and absorbing the constant \(\frac12\) into \(C\), we obtain
\begin{align*}
\frac{d}{dt}I_1(t)
&\leq-\frac{G_0}{4}\|u_t(t)\|^2+\varepsilon_1\left\|A^{\frac12}u(t)\right\|^2+C\|u_t(t)\|^{2\rho+2}\\
&\quad+\hat C(\varepsilon_1)M(\delta)\int_{0}^{t}K_\delta(t-s)\left\|A^{\frac12}u(t)-A^{\frac12}u(s)\right\|^2ds\\
&\quad-\frac{g(0)}{\lambda_1G_0}\int_{0}^{t}g'(t-s)\left\|A^{\frac12}u(t)-A^{\frac12}u(s)\right\|^2ds,
\end{align*}
where we used \(g'\le0\), so that \(-g'\ge0\) and the last term is nonnegative. This yields \eqref{I1esti}. The proof is complete.
\end{proof}

\begin{lemma}\label{lemI2}
Assume that (AOP) and (AG) hold. For any \(t\geq0\),
\begin{equation}\label{I2}
\begin{split}
\frac{d}{dt}I_2(t)
&\leq-\frac{\mu_0}{2}\|A^{\frac12}u(t)\|^2+\|u_t(t)\|^2+C\|u_t(t)\|^{2\rho+2}\\
&\quad+CM(\delta)\int_{0}^{t}K_{\delta}(t-s)\left\|A^{\frac12}u(t)-A^{\frac12}u(s)\right\|^2ds,
\end{split}
\end{equation}
where
\[
I_2(t):=\left(u_t(t),u(t)\right),
\]
and \(\mu_0\), \(K_\delta\), and \(M(\delta)\) are given in \eqref{mu_0}, \eqref{K}, and \eqref{M}, respectively. The constant \(C>0\) depends only on \(\mu_0,\lambda_1\).
\end{lemma}

\begin{proof}
From \eqref{abstract}, Young's inequality, and \eqref{embed} with \(s_1=1\) and \(s_2=0\), we deduce that, for any \(\epsilon>0\),
\begin{align*}
\frac{d}{dt}I_2(t)
&=\left(u_{tt}(t),u(t)\right)+\|u_t(t)\|^2\\
&=-\left(u(t),Au(t)\right)\\
&\quad+\left(\int_{0}^{t}g(t-s)Au(s)ds,u(t)\right)-\left(\|u_t(t)\|^{\rho}u_t(t),u(t)\right)+\|u_t(t)\|^2\\
&=-\left\|A^{\frac12}u(t)\right\|^2\\
&\quad+\left(\int_{0}^{t}g(t-s)A^{\frac12}u(s)ds,A^{\frac12}u(t)\right)-\left(\|u_t(t)\|^{\rho}u_t(t),u(t)\right)+\|u_t(t)\|^2\\
&=-\left\|A^{\frac12}u(t)\right\|^2-\left(\|u_t(t)\|^{\rho}u_t(t),u(t)\right)+\|u_t(t)\|^2\\
&\quad+\left(\int_{0}^{t}g(t-s)(A^{\frac12}u(s)-A^{\frac12}u(t))ds,A^{\frac12}u(t)\right)+\int_{0}^{t}g(s)ds\|A^{\frac12}u(t)\|^2.
\end{align*}
For the nonlocal damping term, by \(\|u(t)\|\le\lambda_1^{-1/2}\|A^{1/2}u(t)\|\) and Young's inequality,
\[
\left|\left(\|u_t(t)\|^{\rho}u_t(t),u(t)\right)\right|
\le \epsilon\|A^{\frac12}u(t)\|^2+C(\epsilon)\|u_t(t)\|^{2\rho+2},
\]
where \(C(\epsilon)\) depends only on \(\epsilon,\lambda_1\). For the memory term, by Young's inequality,
\begin{align*}
&\left|\left(\int_{0}^{t}g(t-s)(A^{\frac12}u(s)-A^{\frac12}u(t))ds,A^{\frac12}u(t)\right)\right|\\
&\qquad \le \epsilon\|A^{\frac12}u(t)\|^2+C(\epsilon)\left\|\int_{0}^{t}g(t-s)(A^{\frac12}u(t)-A^{\frac12}u(s))ds\right\|^2.
\end{align*}
Combining these estimates, we obtain
\begin{align*}
\frac{d}{dt}I_2(t)
&\leq-\left(1-\int_0^t g(s)ds-2\epsilon\right)\left\|A^{\frac12}u(t)\right\|^2+\|u_t(t)\|^2\\
&\quad+C(\epsilon)\|u_t(t)\|^{2\rho+2}
+C(\epsilon)\left\|\int_{0}^{t}g(t-s)(A^{\frac12}u(t)-A^{\frac12}u(s))ds\right\|^2.
\end{align*}
By \eqref{mu_0}, we may choose \(\epsilon=\frac{\mu_0}{8}\) so that
\[
1-\int_0^t g(s)ds-2\epsilon\geq\mu_0-2\epsilon\geq\frac{3\mu_0}{4}\geq\frac{\mu_0}{2}.
\]
Combining this with \eqref{Kd}, which gives
\[
\left\|\int_0^t g(t-s)(A^{1/2}u(t)-A^{1/2}u(s))ds\right\|^2
\le M(\delta)\int_0^t K_\delta(t-s)\left\|A^{1/2}u(t)-A^{1/2}u(s)\right\|^2ds,
\]
we obtain \eqref{I2}, with a constant \(C\) depending only on \(\mu_0,\lambda_1\).
\end{proof}

\begin{lemma}\label{lem2.4}
Assume that (AOP) and (AG) hold. The functional
\begin{align}
I_3(t):=\int_{0}^{t}G(t-s)\left\|A^{\frac12}u(s)\right\|^2ds,\label{I3t}
\end{align}
with \(G(t)=\int_{t}^{\infty}g(s)ds\) defined in \eqref{G}, satisfies
\begin{align}
\frac{d}{dt}I_3(t)=-\int_{0}^{t}g(t-s)\left\|A^{\frac12}u(s)\right\|^2ds+G_0\left\|A^{\frac12}u(t)\right\|^2,\label{I3d}
\end{align}
and
\begin{align}
\frac{d}{dt}I_3(t)\leq-\frac12\int_{0}^{t}g(t-s)\left\|A^{\frac12}u(t)-A^{\frac12}u(s)\right\|^2ds+3G_0\left\|A^{\frac12}u(t)\right\|^2,\label{I3}
\end{align}
where \(G_0\) is given in \eqref{mu_0}.
\end{lemma}

\begin{proof}
Since \(G'(t-s)=-g(t-s)\) and \(G(0)=\int_0^{\infty}g(s)ds=G_0\), identity \eqref{I3d} follows directly.

To prove \eqref{I3}, we rewrite \eqref{I3d} as
\begin{align*}
\frac{d}{dt}I_3(t)
&=-\int_{0}^{t}g(t-s)\left\|A^{\frac12}u(t)-A^{\frac12}u(s)\right\|^2ds+\int_{0}^{t}g(s)ds\left\|A^{\frac12}u(t)\right\|^2\\
&\quad-2\left(\int_{0}^{t}g(t-s)A^{\frac12}u(s)ds,A^{\frac12}u(t)\right)+G_0\left\|A^{\frac12}u(t)\right\|^2.
\end{align*}
Using
\[
G_0=\int_{0}^{\infty}g(s)ds=\int_{0}^{t}g(s)ds+G(t),
\]
we obtain
\begin{align*}
\frac{d}{dt}I_3(t)
&=-\int_{0}^{t}g(t-s)\left\|A^{\frac12}u(t)-A^{\frac12}u(s)\right\|^2ds+G(t)\left\|A^{\frac12}u(t)\right\|^2\notag\\
&\quad-2\left(A^{\frac12}u(t),\int_{0}^{t}g(t-s)(A^{\frac12}u(s)-A^{\frac12}u(t))ds\right).
\end{align*}
By Young's inequality and H\"older's inequality,
\begin{align*}
&-2\left(A^{\frac12}u(t),\int_{0}^{t}\sqrt{g(t-s)}\sqrt{g(t-s)}\left(A^{\frac12}u(s)-A^{\frac12}u(t)\right)ds\right)\\
&\leq2G_0\left\|A^{\frac12}u(t)\right\|^2+\frac1{2G_0}\underbrace{\int_0^tg(s)ds}_{\leq G_0}\int_{0}^{t}g(t-s)\left\|A^{\frac12}u(t)-A^{\frac12}u(s)\right\|^2ds\\
&\leq2G_0\left\|A^{\frac12}u(t)\right\|^2+\frac{1}{2}\int_{0}^{t}g(t-s)\left\|A^{\frac12}u(t)-A^{\frac12}u(s)\right\|^2ds.
\end{align*}
Since \(G(t)\leq G_0\), estimate \eqref{I3} follows.
\end{proof}

\begin{lemma}\label{lem2.5}
Assume that (AOP) and (AG) hold. The functional
\[
I_4(t):=M(\delta)\left(\delta\int_{0}^{t}G(t-s)\|A^{\frac12}u(s)\|^2ds+E(t)\right)
\]
satisfies
\begin{align}
\frac{d}{dt}I_4(t)\leq-\frac12M(\delta)\int_{0}^{t}K_{\delta}(t-s)\left\|A^{\frac12}u(t)-A^{\frac12}u(s)\right\|^2ds\notag\\
+3G_0\delta M(\delta)\|A^{\frac12}u(t)\|^2,\label{I4}
\end{align}
where \(E\) is defined in \eqref{energy}, and \(G_0\), \(K_\delta\), and \(M(\delta)\) are given in \eqref{mu_0}, \eqref{K}, and \eqref{M}, respectively.
\end{lemma}

\begin{proof}
From \eqref{energy}, we have
\begin{equation}\label{ed}
E'(t)=-\frac12g(t)\left\|A^{\frac12}u\right\|^2-\|u_t\|^{\rho+2}+\frac12\int_0^tg'(t-s)\left\|A^{\frac12}u(t)-A^{\frac12}u(s)\right\|^2ds,
\end{equation}
which implies
\[
E'(t)\leq\frac12\int_{0}^{t}g'(t-s)\left\|A^{\frac12}u(t)-A^{\frac12}u(s)\right\|^2ds.
\]
Combining this with \eqref{I3}, we obtain
\begin{align*}
\frac{d}{dt}I_4(t)
&=M(\delta)(\delta I'_3(t)+E'(t))\\
&\leq-\frac{M(\delta)}{2}\int_{0}^{t}\delta g(t-s)\left\|A^{\frac12}u(t)-A^{\frac12}u(s)\right\|^2ds+3M(\delta)\delta G_0\left\|A^{\frac12}u(t)\right\|^2\\
&\quad+\frac{M(\delta)}{2}\int_{0}^{t}g'(t-s)\left\|A^{\frac12}u(t)-A^{\frac12}u(s)\right\|^2ds\\
&=-\frac12M(\delta)\int_{0}^{t}K_\delta(t-s)\left\|A^{\frac12}u(t)-A^{\frac12}u(s)\right\|^2ds+3M(\delta)\delta G_0\left\|A^{\frac12}u(t)\right\|^2,
\end{align*}
where we used \(K_\delta(t-s)=\delta g(t-s)-g'(t-s)\). This proves \eqref{I4}.
\end{proof}

\begin{proof}[Proof of Theorem \ref{main} for the case \(-1<\rho<0\)]
Define
\[
J(t):=N_1I_1(t)+N_2I_2(t)+I_3(t)+N_4I_4(t),
\]
where \(I_i\), \(i=1,2,3,4\), are the functionals defined above, and \(N_1,N_2,N_4\) are positive constants to be determined.

Using \eqref{ed}, \eqref{I1esti}, \eqref{I2}, \eqref{I3}, and \eqref{I4}, we obtain, for \(t\geq t_0\),
\begin{align*}
\frac{d}{dt}J(t)
&\leq\left(N_1\varepsilon_1-\frac{N_2\mu_0}{2}+3G_0+3N_4G_0\delta M(\delta)\right)\left\|A^{\frac12}u(t)\right\|^2
+\left(-\frac{N_1G_0}{4}+N_2\right)\|u_t(t)\|^2\\
&\quad+\left(N_1\hat C(\varepsilon_1)M(\delta)+N_2CM(\delta)-\frac{N_4M(\delta)}{2}\right)\int_{0}^{t}K_{\delta}(t-s)\left\|A^{\frac12}u(t)-A^{\frac12}u(s)\right\|^2ds\\
&\quad-\frac12\int_{0}^{t}g(t-s)\left\|A^{\frac12}u(t)-A^{\frac12}u(s)\right\|^2ds+(N_1+N_2)C\|u_t(t)\|^{2\rho+2}\\
&\quad-N_1\frac{g(0)}{\lambda_1G_0}\int_{0}^{t}g'(t-s)\left\|A^{\frac12}u(t)-A^{\frac12}u(s)\right\|^2ds.
\end{align*}
Here the coefficient of \(\|A^{1/2}u\|^2\) is obtained by summing the contributions
\[
N_1\varepsilon_1,\qquad -\frac{N_2\mu_0}{2},\qquad 3G_0,\qquad 3N_4G_0\delta M(\delta),
\]
from \(I_1,I_2,I_3,I_4\), respectively.

By (AG), for almost every \(s\in(0,\infty)\),
\[
\frac{\delta g^2(s)}{-g'(s)+\delta g(s)}\leq g(s)\in L^1(0,\infty).
\]
Hence, by \eqref{K}, \eqref{M}, and the dominated convergence theorem,
\[
\lim_{\delta\to 0}\delta M(\delta)=\lim_{\delta\to 0}\int_0^{\infty}\frac{\delta g^2(s)}{-g'(s)+\delta g(s)}ds=0.
\]
We may therefore choose \(\delta\) small enough such that
\begin{equation}\label{deltamdelta}
3N_4G_0\delta M(\delta)\leq\frac{N_2\mu_0}{8}.
\end{equation}
Moreover, choosing
\[
\varepsilon_1=\frac{N_2\mu_0}{8N_1},\qquad
N_2=\frac{48G_0}{\mu_0},\qquad
N_1=\frac{4N_2+1}{G_0},\qquad
N_4=2(N_1\hat C(\varepsilon_1)+N_2C)+1,
\]
we obtain
\[
N_1\varepsilon_1=\frac{N_2\mu_0}{8}=6G_0,\qquad
3N_4G_0\delta M(\delta)\leq\frac{N_2\mu_0}{8}=6G_0.
\]
Hence
\[
\frac{N_2\mu_0}{2}-N_1\varepsilon_1-3G_0-3N_4G_0\delta M(\delta)
\geq24G_0-6G_0-3G_0-6G_0=9G_0>0.
\]
Therefore, for \(t\geq t_0\),
\begin{align*}
\frac{d}{dt}J(t)
&\leq-9G_0\|A^{\frac12}u(t)\|^2-\frac14\|u_t(t)\|^2\\
&\quad-\frac12\int_{0}^{t}g(t-s)\left\|A^{\frac12}u(t)-A^{\frac12}u(s)\right\|^2ds+(N_1+N_2)C\|u_t(t)\|^{2\rho+2}\\
&\quad-N_1\frac{g(0)}{\lambda_1G_0}\int_{0}^{t}g'(t-s)\left\|A^{\frac12}u(t)-A^{\frac12}u(s)\right\|^2ds.
\end{align*}
Since
\[
E(t)\geq\frac12\|u_t\|^2+\frac{\mu_0}{2}\|A^{\frac12}u\|^2+\frac12\int_{0}^{t}g(t-s)\left\|A^{\frac12}u(t)-A^{\frac12}u(s)\right\|^2ds,
\]
\[
-E'(t)\geq\|u_t(t)\|^{\rho+2},\qquad
-E'(t)\geq-\frac12\int_{0}^{t}g'(t-s)\left\|A^{\frac12}u(t)-A^{\frac12}u(s)\right\|^2ds,
\]
and, since \(-1<\rho<0\), we have \(\rho+2\in(1,2)\), so that
\[
\|u_t(t)\|^{2\rho+2}\le C(-E'(t))^{\frac{2\rho+2}{\rho+2}}.
\]
Indeed, by \eqref{ed} and \(g'\le0\),
\[
-E'(t)\ge\|u_t(t)\|^{\rho+2},
\]
and raising both sides to the power \(\frac{2\rho+2}{\rho+2}>0\) gives the claim. Hence we deduce, for \(t\geq t_0\),
\begin{align}
\frac{d}{dt}J(t)
&\leq-CE(t)-C\int_{0}^{t}g(t-s)\left\|A^{\frac12}u(s)\right\|^2ds-CE'(t)+C(-E'(t))^{\frac{2\rho+2}{\rho+2}}.\label{J}
\end{align}
Here the constant \(C>0\) depends only on \(\rho,\mu_0,\lambda_1,G_0,g(0),N_1,N_2,N_4,\hat C(\varepsilon_1)\).

Set
\[
L(t):=NE(t)+J(t)=NE(t)+N_1I_1(t)+N_2I_2(t)+I_3(t)+N_4I_4(t),
\]
where \(N\) is a positive constant to be determined. By the definitions of \(E\) and \(I_1\), Young's inequality, H\"older's inequality, and \eqref{embed} with \(s_1=1\) and \(s_2=0\), we have
\begin{align*}
|I_1(t)|
&=\left|\left(u_t(t),\int_{0}^{t}g(t-s)(u(t)-u(s))ds\right)\right|\\
&\leq C\|u_t(t)\|^2+C\left\|\int_{0}^{t}\sqrt{g(t-s)}\sqrt{g(t-s)}\left(u(t)-u(s)\right)ds\right\|^2\\
&\leq C\|u_t(t)\|^2+C\underbrace{\int_0^{\infty}g(s)ds}_{=G_0}\int_{0}^{t}g(t-s)\left\|A^{\frac12}u(t)-A^{\frac12}u(s)\right\|^2ds\\
&\leq CE(t),
\end{align*}
where \(C\) depends only on \(G_0,\lambda_1\). Similarly, \(|I_2(t)|\leq CE(t)\), with \(C\) depending only on \(\lambda_1\). Moreover, \(I_3(t),I_4(t)\geq0\). Hence, choosing \(N\) sufficiently large, we obtain
\begin{equation}\label{Lda}
L(t)\geq CE(t),
\end{equation}
where the constant \(C>0\) depends on \(N_1,N_2,G_0,\lambda_1\).

By (AG), \(\lim_{t\to\infty}g(t)=0\), and \(g'(t)\leq-kg(t)\) for almost every \(t\in[0,\infty)\). Thus, by \eqref{G},
\[
G(t)=\int_t^{\infty}g(s)ds\leq-\frac{1}{k}\int_t^{\infty}g'(s)ds=\frac1k g(t).
\]
Using the definitions of \(I_3\) and \(I_4\), together with \eqref{deltamdelta}, we get
\[
I_3(t)\leq C\int_{0}^{t}g(t-s)\left\|A^{\frac12}u(s)\right\|^2ds,\qquad
I_4(t)\leq CE(t)+C\int_{0}^{t}g(t-s)\left\|A^{\frac12}u(s)\right\|^2ds.
\]
Since
\begin{align*}
\int_{0}^{t}g(t-s)\left\|A^{\frac12}u(s)\right\|^2ds
&=\int_{0}^{t}g(t-s)\left\|A^{\frac12}u(s)-A^{\frac12}u(t)+A^{\frac12}u(t)\right\|^2ds\\
&\leq2\int_{0}^{t}g(t-s)\left\|A^{\frac12}u(s)-A^{\frac12}u(t)\right\|^2ds\\
&\quad+2\int_{0}^{t}g(t-s)ds\left\|A^{\frac12}u(t)\right\|^2\\
&\leq CE(t),
\end{align*}
we obtain
\begin{align}
C_1E(t)\leq L(t)\leq C_2E(t)\label{dj}
\end{align}
for some constants \(C_1,C_2>0\) depending only on \(\rho,\mu_0,\lambda_1,G_0,g(0),k\) and on \(M(\delta)\).

From \eqref{J} and \eqref{dj}, for \(t\geq t_0\), and choosing \(N\) sufficiently large so that \(N-C\geq0\), we get
\begin{align}
\frac{d}{dt}L(t)
&=NE'(t)+J'(t)\notag\\
&\leq-CE(t)-C\int_{0}^{t}g(t-s)\left\|A^{\frac12}u(s)\right\|^2ds\notag\\
&\quad+(N-C)E'(t)+C(-E'(t))^{\frac{2\rho+2}{\rho+2}}\notag\\
&\leq-CL(t)+C(-E'(t))^{\frac{2\rho+2}{\rho+2}}.\label{betterrs}
\end{align}
Since \(-1<\rho<0\), multiplying \eqref{betterrs} by \(L^{\frac{-\rho}{2\rho+2}}(t)\) and integrating over \([t,T]\) with \(T>t\geq t_0\), we obtain, by Young's inequality,
\begin{align*}
\frac{2\rho+2}{\rho+2}\left(L^{\frac{\rho+2}{2\rho+2}}(T)-L^{\frac{\rho+2}{2\rho+2}}(t)\right)
&\leq-C\int_{t}^{T}L^{\frac{\rho+2}{2\rho+2}}(s)ds+C\int_{t}^{T}L^{\frac{-\rho}{2\rho+2}}(s)(-E'(s))^{\frac{2\rho+2}{\rho+2}}ds\\
&\leq-C\int_{t}^{T}L^{\frac{\rho+2}{2\rho+2}}(s)ds\\
&\quad+\frac{C}{2}\int_{t}^{T}L^{\frac{\rho+2}{2\rho+2}}(s)ds+\mathcal{C}(\rho)\int_{t}^{T}(-E'(s))ds,
\end{align*}
which, together with \eqref{dj}, implies
\begin{align*}
\int_{t}^{T}L^{\frac{\rho+2}{2\rho+2}}(s)ds
&\leq CL^{\frac{\rho+2}{2\rho+2}}(t)+C\int_{t}^{T}(-E'(s))ds\\
&\leq CL^{\frac{\rho+2}{2\rho+2}}(t)+CE(t)
\leq C\left(L^{\frac{\rho+2}{2\rho+2}-1}(t)+1\right)L(t)
\\
&=C\left(L^{\frac{-\rho}{2\rho+2}}(t)+1\right)L(t).
\end{align*}
Hence, since \(L(t)\leq C_2E(t)\leq C_2E(0)\),
\[
\int_{t}^{T}L^{\frac{\rho+2}{2\rho+2}}(s)ds\leq C\left(1+E(0)^{\frac{-\rho}{2\rho+2}}\right)L(t).
\]
By \eqref{dj},
\begin{align}
\int_{t}^{T}E^{\frac{\rho+2}{2\rho+2}}(s)ds\leq C\left(1+E(0)^{\frac{-\rho}{2\rho+2}}\right)E(t)\leq C\left(1+E(0)^{\frac{-\rho}{2\rho+2}}\right)E(0).\label{case1plo}
\end{align}
Taking \(T=2t\), and using \(E'(t)\leq0\), we get
\[
[E(2t)]^{\frac{\rho+2}{2\rho+2}}t\leq C\left(1+E(0)^{\frac{-\rho}{2\rho+2}}\right)E(0),
\]
so that
\[
E(2t)\leq(2C)^{\frac{2\rho+2}{\rho+2}}\left(1+E(0)^{\frac{-\rho}{2\rho+2}}\right)^{\frac{2\rho+2}{\rho+2}}E(0)^{\frac{2\rho+2}{\rho+2}}(2t)^{-\frac{2\rho+2}{\rho+2}}.
\]
Repeating the above iteration \(m\) times yields
\begin{align}
E(t)\leq t^{-\sum\limits_{i=1}^{m}\left(\frac{2\rho+2}{\rho+2}\right)^i}E^{\left(\frac{2\rho+2}{\rho+2}\right)^m}(0)C^{\sum\limits_{i=1}^{m}\left(\frac{2\rho+2}{\rho+2}\right)^i}\left(1+E(0)^{\frac{-\rho}{2\rho+2}}\right)^{\sum\limits_{i=1}^{m}\left(\frac{2\rho+2}{\rho+2}\right)^i}2^{\sum\limits_{i=1}^{m}i\left(\frac{2\rho+2}{\rho+2}\right)^i}.\label{case1plo01}
\end{align}
Let
\[
q:=\frac{2\rho+2}{\rho+2}.
\]
Since \(-1<\rho<0\), we have \(0<q<1\). Therefore, as \(m\to\infty\),
\[
\sum\limits_{i=1}^{m}q^i\to\frac{q}{1-q}=-\frac{2\rho+2}{\rho},\qquad
q^m\to0,\qquad
\sum\limits_{i=1}^{m}iq^i\to\frac{q}{(1-q)^2}=\frac{(2\rho+2)(\rho+2)}{\rho^2}.
\]
Letting \(m\to\infty\) in \eqref{case1plo01}, we conclude
\[
E(t)\leq t^{\frac{2\rho+2}{\rho}}C^{-\frac{2\rho+2}{\rho}}\left(1+E(0)^{\frac{-\rho}{2\rho+2}}\right)^{-\frac{2\rho+2}{\rho}}2^{\frac{(2\rho+2)(\rho+2)}{\rho^2}}.
\]
Thus there exists a constant
\[
C(\rho,\mu_0,\lambda_1,g(0),G_0,k)>0
\]
such that
\[
E(t)\leq C(\rho,\mu_0,\lambda_1,g(0),G_0,k)
\left(1+E(0)^{\frac{-\rho}{2\rho+2}}\right)^{-\frac{2\rho+2}{\rho}}
t^{\frac{2\rho+2}{\rho}},\qquad t\ge t_0.
\]
Since \(-\frac{2\rho+2}{\rho}=\frac{2\rho+2}{|\rho|}>0\) and
\(-\frac{\rho}{2\rho+2}\in(0,\frac12)\), we have
\[
E(0)^{\frac{-\rho}{2\rho+2}}\le 1+E(0)^{-1},
\]
and hence
\[
\left(1+E(0)^{-\frac{\rho}{2\rho+2}}\right)^{\frac{2\rho+2}{|\rho|}}
\le C(\rho)\left(1+E(0)^{-1}\right)^{\frac{2\rho+2}{|\rho|}}.
\]
Consequently,
\[
E(t)\leq C(\rho,\mu_0,\lambda_1,g(0),G_0,k)
\left(1+E(0)^{-1}\right)^{\frac{2\rho+2}{|\rho|}}
t^{\frac{2\rho+2}{\rho}},\qquad t\ge t_0.
\]
This completes the proof for \(-1<\rho<0\).
\end{proof}

\noindent\textbf{The case $\bm{\rho\geq0}$.} In this part we prove Theorem \ref{main} for \(\rho\geq0\). We shall use the following standard lemma, which can be found in \cite{MR0499564,MR0513082}.

\begin{lemma}\label{lemmp}
Let \(\phi:[0,\infty)\to\mathbb{R}\) be a nonnegative and bounded function satisfying
\[
\sup_{t\leq s\leq t+1}[\phi(s)]^{1+\sigma}\leq c_0[\phi(t)-\phi(t+1)],\qquad t\geq0,
\]
where \(c_0>0\) and \(\sigma\geq0\) are constants. Then:
\begin{description}
\item[(i)] if \(\sigma>0\),
\[
\phi(t)\leq\left[\left(\sup_{0\leq s\leq1}\phi(s)\right)^{-\sigma}+\frac{\sigma}{c_0}(t-1)^+\right]^{-1/\sigma},\qquad t\geq0;
\]
\item[(ii)] if \(\sigma=0\),
\[
\phi(t)\leq\left(\sup_{0\leq s\leq1}\phi(s)\right)e^{-\kappa(t-1)^+},\qquad t\geq0,
\]
where \(\kappa=-\log(1-1/c_0)\). Here \(s^+:=\frac12(s+|s|)\).
\end{description}
\end{lemma}

\begin{proof}[Proof of Theorem \ref{main} for the case \(\rho\geq0\)]
From (AG) and \eqref{ed}, we infer that
\begin{align}
\int_{t}^{t+1}\|u_t(s)\|^{\rho+2}ds
&\leq\int_{t}^{t+1}\left(\frac12g(s)\left\|A^{\frac12}u(s)\right\|^2+\|u_t(s)\|^{\rho+2}\right.\notag\\
&\quad\left.-\frac12\int_{0}^{s}g'(s-\tau)\left\|A^{\frac12}u(s)-A^{\frac12}u(\tau)\right\|^2d\tau\right)ds\notag\\
&=\int_{t}^{t+1}-E'(s)ds=E(t)-E(t+1):=[D(t)]^2.\label{nba}
\end{align}
By H\"older's inequality with \(\frac{\rho}{\rho+2}+\frac{2}{\rho+2}=1\),
\begin{align}\label{klll}
\int_{t}^{t+1}\|u_t(s)\|^{2}ds
&\leq\left(\int_{t}^{t+1}ds\right)^{\frac{\rho}{\rho+2}}\left(\int_{t}^{t+1}\|u_t(s)\|^{\rho+2}ds\right)^{\frac{2}{\rho+2}}
\leq[D(t)]^{\frac{4}{\rho+2}}.
\end{align}
By the mean value theorem, there exist \(t_1\in[t,t+\frac14]\) and \(t_2\in[t+\frac34,t+1]\) such that
\begin{align}
\|u_t(t_i)\|^2\leq4\int_{t}^{t+1}\|u_t(s)\|^{2}ds\leq4[D(t)]^{\frac{4}{\rho+2}},\qquad i=1,2.\label{mean}
\end{align}

Next, we consider the memory term. We have the decomposition
\begin{align}\label{cond1}
-\left(\int_{0}^{t}g(t-s)A^{\frac12}u(s)ds,A^{\frac12}u(t)\right)
&=\frac12\int_{0}^{t}g(t-s)\left\|A^{\frac12}u(t)-A^{\frac12}u(s)\right\|^2ds\notag\\
&\quad-\frac12\int_{0}^{t}g(s)ds\left\|A^{\frac12}u(t)\right\|^2\notag\\
&\quad-\frac12\int_{0}^{t}g(t-s)\left\|A^{\frac12}u(s)\right\|^2ds.
\end{align}
For the last term on the right-hand side of \eqref{cond1}, by (AG), \eqref{ed}, and \((a+b)^2\leq2a^2+2b^2\), we obtain
\begin{align}\label{cond2}
\frac12\int_{0}^{t}g(t-s)\left\|A^{\frac12}u(s)\right\|^2ds
&\leq-\frac{1}{2k}\int_{0}^{t}g'(t-s)\left\|A^{\frac12}u(s)\right\|^2ds\notag\\
&\leq-\frac{1}{k}\int_{0}^{t}g'(t-s)\left\|A^{\frac12}u(t)-A^{\frac12}u(s)\right\|^2ds\notag\\
&\quad+\frac{1}{k}\int_{0}^{t}-g'(t-s)ds\left\|A^{\frac12}u(t)\right\|^2\notag\\
&\leq\frac2k(-E'(t))+\frac{g(0)}{k}\left\|A^{\frac12}u(t)\right\|^2.
\end{align}
Multiplying \eqref{abstract} by \(u(t)\) and integrating over \(V_0\times(t_1,t_2)\), we get
\begin{align*}
\int_{t_1}^{t_2}\left(u_{ss}(s),u(s)\right)ds+\int_{t_1}^{t_2}\left[\left\|A^{\frac12}u(s)\right\|^2-\left(\int_{0}^{s}g(s-\tau)A^{\frac12}u(\tau)d\tau,A^{\frac12}u(s)\right)\right]ds\\
+\int_{t_1}^{t_2}\left(\|u_s(s)\|^{\rho}u_s(s),u(s)\right)ds=0.
\end{align*}
Since \(\frac{d}{dt}\left(u_{t}(t),u(t)\right)=\left(u_{tt}(t),u(t)\right)+\|u_t(t)\|^2\), this yields
\begin{align}
\int_{t_1}^{t_2}\left[\|A^{\frac12}u(s)\|^2-\left(\int_{0}^{s}g(s-\tau)A^{\frac12}u(\tau)d\tau,A^{\frac12}u(s)\right)\right]ds\notag\\
=\int_{t_1}^{t_2}\|u_s(s)\|^2ds+\sum\limits_{i=1}^{2}J_i,\label{key}
\end{align}
where
\[
J_1=(u(t_1),u_t(t_1))-(u(t_2),u_t(t_2)),\qquad
J_2=-\int_{t_1}^{t_2}\left(\|u_s(s)\|^{\rho}u_s(s),u(s)\right)ds.
\]
Replacing the second term on the left-hand side of \eqref{key} with \eqref{cond1}, we obtain
\begin{align*}
&\int_{t_1}^{t_2}\left(1-\frac12\int_{0}^{s}g(\tau)d\tau\right)\|A^{\frac12}u(s)\|^2ds+\frac12\int_{t_1}^{t_2}\int_{0}^{s}g(s-\tau)\|A^{\frac12}u(s)-A^{\frac12}u(\tau)\|^2d\tau ds\\
&\quad-\frac12\int_{t_1}^{t_2}\int_{0}^{s}g(s-\tau)\|A^{\frac12}u(\tau)\|^2d\tau ds=\int_{t_1}^{t_2}\|u_s(s)\|^2ds+\sum\limits_{i=1}^{2}J_i.
\end{align*}
Combining this with \eqref{cond2}, we get
\begin{align*}
&\int_{t_1}^{t_2}\left(1-\frac12\int_{0}^{s}g(\tau)d\tau-\frac{g(0)}{k}\right)\|A^{\frac12}u(s)\|^2ds\\
&\quad+\frac12\int_{t_1}^{t_2}\int_{0}^{s}g(s-\tau)\|A^{\frac12}u(s)-A^{\frac12}u(\tau)\|^2d\tau ds\\
&\leq\int_{t_1}^{t_2}\|u_s(s)\|^2ds+\sum\limits_{i=1}^{2}J_i+\frac2k\int_{t_1}^{t_2}(-E'(s))ds.
\end{align*}
Since \(k\geq2g(0)\), we have
\[
1-\frac12\int_0^sg(\tau)d\tau-\frac{g(0)}{k}\geq\frac12\left(1-\int_0^sg(\tau)d\tau\right).
\]
Thus, using \(E'(t)\leq0\), \(t_1\in[t,t+\frac14]\), and \(t_2\in[t+\frac34,t+1]\), we obtain
\begin{align*}
\int_{t_1}^{t_2}E(s)ds
&\leq\frac32\int_{t_1}^{t_2}\|u_s(s)\|^2ds+\sum\limits_{i=1}^{2}J_i+\frac2k\int_{t}^{t+1}(-E'(s))ds\\
&\leq\frac32\int_{t_1}^{t_2}\|u_s(s)\|^2ds+\sum\limits_{i=1}^{2}J_i+\frac2k[D(t)]^2,
\end{align*}
that is,
\begin{equation}\label{mid}
\int_{t_1}^{t_2}E(s)ds\leq\frac32\int_{t_1}^{t_2}\|u_s(s)\|^2ds+\sum\limits_{i=1}^{2}J_i+\frac2k[D(t)]^2.
\end{equation}

Since \(t\leq t_1<t_2\leq t+1\), by \eqref{klll},
\begin{equation}\label{klkl}
\int_{t_1}^{t_2}\|u_s(s)\|^2ds\leq\int_{t}^{t+1}\|u_s(s)\|^2ds\leq[D(t)]^{\frac{4}{\rho+2}}.
\end{equation}
Using \eqref{embed} with \(s_1=1\) and \(s_2=0\), we have \(\|u(t_i)\|\leq C_1\|A^{\frac12}u(t_i)\|\) with \(C_1:=\lambda_1^{-1/2}\). Hence, by \eqref{mu_0}, \eqref{energy}, and \(E(t)\geq E(t_i)\) for \(t_i\geq t\),
\[
\|u(t_i)\|\leq C_1\left\|A^{\frac12}u(t_i)\right\|\leq C_1\sqrt{\frac{2}{\mu_0}E(t_i)}\leq C_1\sqrt{\frac{2}{\mu_0}E(t)}.
\]
Together with \eqref{mean}, this gives
\[
J_1^2\leq2\left(\|u(t_1)\|^2\|u_t(t_1)\|^2+\|u(t_2)\|^2\|u_t(t_2)\|^2\right)\leq32C_1^2\mu_0^{-1}E(t)[D(t)]^{\frac{4}{\rho+2}},
\]
and hence, by Young's inequality,
\[
J_1\leq\epsilon E(t)+\frac{8C_1^2}{\epsilon\mu_0}[D(t)]^{\frac{4}{\rho+2}}
\]
for any \(\epsilon>0\).

By \eqref{energy} and \(E'(t)\leq0\), we have \(\|u_t\|^2\leq2E(0)\). Let \(C_2:=(2E(0))^{\frac{\rho}{2}}\). By H\"older's inequality, \eqref{klkl}, and Young's inequality, we obtain
\begin{align*}
J_2
&\leq C_2\int_{t_1}^{t_2}\|u(s)\|\|u_s(s)\|ds
\leq C_1C_2\int_{t_1}^{t_2}\|A^{\frac12}u(s)\|\|u_s(s)\|ds\\
&\leq C_1C_2\sqrt{2}\mu_0^{-\frac12}[E(t)]^{\frac12}\int_{t_1}^{t_2}\|u_s(s)\|ds
\leq C_1C_2\sqrt{2}\mu_0^{-\frac12}[E(t)]^{\frac12}\left(\int_{t_1}^{t_2}\|u_s(s)\|^2ds\right)^{\frac12}\\
&\leq C_1C_2\sqrt{2}\mu_0^{-\frac12}[E(t)]^{\frac12}[D(t)]^{\frac{2}{\rho+2}}
\leq\epsilon E(t)+\frac{C_1^2C_2^2}{2\epsilon\mu_0}[D(t)]^{\frac{4}{\rho+2}}
\end{align*}
for any \(\epsilon>0\).

Combining the estimates for \(J_1\) and \(J_2\) with \eqref{klkl} and \eqref{mid}, we get
\[
\int_{t_1}^{t_2}E(s)ds\leq2\epsilon E(t)+\left(\frac{8C_1^2}{\epsilon\mu_0}+\frac{C_1^2C_2^2}{2\epsilon\mu_0}+\frac32\right)[D(t)]^{\frac{4}{\rho+2}}+\frac2k[D(t)]^2.
\]
Since \([D(t)]^2\leq E(t)\) and \(E'(t)\leq0\), we have
\begin{equation}\label{lko}
[D(t)]^2=[D(t)]^{\frac{4}{\rho+2}}[D(t)]^{\frac{2\rho}{\rho+2}}\leq E(0)^{\frac{\rho}{\rho+2}}[D(t)]^{\frac{4}{\rho+2}}.
\end{equation}
Thus
\begin{equation}\label{next}
\int_{t_1}^{t_2}E(s)ds\leq2\epsilon E(t)+\mu_1[D(t)]^{\frac{4}{\rho+2}}
\end{equation}
for any \(\epsilon>0\), where
\[
\mu_1:=\frac{8C_1^2}{\epsilon\mu_0}+\frac{C_1^2C_2^2}{2\epsilon\mu_0}+\frac32+\frac2kE(0)^{\frac{\rho}{\rho+2}}.
\]
Since \(E'(t)\leq0\), the mean value theorem gives some \(\beta\in[t_1,t_2]\subset[t,t+1]\) such that
\begin{align}
\int_{t_1}^{t_2}E(s)ds=E(\beta)(t_2-t_1)\geq\frac12E(t+1).\label{most}
\end{align}
Combining \eqref{nba}, \eqref{next}, \eqref{most}, and \eqref{lko}, we obtain
\begin{align*}
E(t)
&=[D(t)]^2+E(t+1)\leq[D(t)]^2+2\int_{t_1}^{t_2}E(s)ds\\
&\leq[D(t)]^2+2\mu_1[D(t)]^{\frac{4}{\rho+2}}+4\epsilon E(t)\\
&\leq\left[E(0)^{\frac{\rho}{\rho+2}}+2\mu_1\right][D(t)]^{\frac{4}{\rho+2}}+4\epsilon E(t).
\end{align*}
Taking \(\epsilon=\frac18\), we get
\[
E(t)\leq2\left[E(0)^{\frac{\rho}{\rho+2}}+2\mu_1\right][D(t)]^{\frac{4}{\rho+2}},
\]
and hence
\[
[E(t)]^{1+\frac{\rho}{2}}\leq\mu[D(t)]^2=\mu[E(t)-E(t+1)],
\]
where
\[
\mu:=\left[2E(0)^{\frac{\rho}{\rho+2}}+4\mu_1\right]^{1+\frac{\rho}{2}}.
\]
Since \(E'(t)\leq0\),
\[
\sup_{t\leq s\leq t+1}[E(t)]^{1+\frac{\rho}{2}}\leq[E(t)]^{1+\frac{\rho}{2}}\leq\mu[E(t)-E(t+1)].
\]
Applying Lemma \ref{lemmp}, we obtain the desired decay estimates. This completes the proof for \(\rho\geq0\).
\end{proof}

\section{Discussion and comparison with the literature}\label{sec5}

We conclude by placing our results in the context of the existing literature.

For \(\rho=0\), Problem \eqref{abstract} reduces to a viscoelastic wave equation with frictional damping. In this case, the exponential decay obtained in Theorem \ref{main} is consistent with the classical results of Cavalcanti and Oquendo \cite{MR2044797} and with subsequent refinements in \cite{MR2542886,MR3702699,MR4232931}. The novelty here is not the qualitative rate itself, but the fact that it is proved in a general abstract Hilbert-space setting, with explicit constants and without additional geometric conditions on the underlying domain. Related abstract and plate models with memory and frictional dissipation can be found in \cite{MR2386941,MR2793256,MR3401605,MR3449383,MR4145756,MR4592221}.

For \(\rho>0\), the term \(\|u_t\|^\rho u_t\) is a nonlocal damping. Equations of this type, without the memory term, were studied by Haraux and coauthors \cite{MR3095203,MR3918410} and by Lour\^edo et al. \cite{MR3214190}; see also \cite{MR4098779,MR4369100} for nonlocal damping with anti-damping, and \cite{MR4341621,MR4342158,MR4064014} for related attractor results. In those works, the typical decay is polynomial for \(\rho>0\) and exponential for \(\rho=0\). Theorem \ref{main} shows that the same qualitative decay rates persist when the viscoelastic memory term is added, at least under the exponential decay assumption (AG) on \(g\). This is not a direct consequence of the memoryless results, because the memory term changes both the energy functional and the multiplier identities.

The case \(-1<\rho<0\) is the main novelty of the paper. Here the damping operator \(v\mapsto \|v\|^\rho v\) is not Lipschitz at the origin, and the classical uniqueness arguments based on Lipschitz estimates do not apply. We overcome this difficulty by using the monotonicity of this operator, which yields uniqueness for the whole range \(\rho>-1\). The decay estimate
\[
E(t)\le C(\rho,\mu_0,\lambda_1,g(0),G_0,k)
\left(1+E(0)^{-1}\right)^{\frac{2\rho+2}{|\rho|}}
t^{\frac{2\rho+2}{\rho}},\qquad t\ge t_0,
\]
is polynomial, since \(2\rho+2>0\) and \(\rho<0\), so the exponent \((2\rho+2)/\rho\) is negative. Here \(t_0>0\) is defined by \(\int_0^{t_0}g(s)ds=\frac12 G_0\), as in \eqref{existencet0}. As \(\rho\to -1^+\), this exponent tends to \(0^-\), and the decay becomes arbitrarily slow; as \(\rho\to 0^-\), the exponent tends to \(-\infty\), so the decay becomes arbitrarily fast. This reflects the singular nature of the nonlocal damping in the negative-exponent regime. To our knowledge, no comparable abstract result is available; the closest works are \cite{MR4611266,MR3986205,MR4577006} for nonlocal damping in beam models, but they do not cover the negative-exponent regime \(-1<\rho<0\) in the abstract setting considered here.

Several limitations and open problems are worth mentioning. The case \(\rho\le -1\) is left open, since the damping operator is no longer continuous and monotone at the origin, and the dissipation rate \(\|u_t\|^{\rho+2}\) may become singular. The assumption (AG) requires exponential decay of the memory kernel; it would be interesting to extend the analysis to more general kernels satisfying \(g'(t)\le -H(g(t))\) or \(g'(t)\le -\xi(t)g(t)\). Finally, the operator \(A\) is assumed to be positive, self-adjoint and with compact resolvent; removing the compactness assumption is another natural direction. These questions will be addressed in future work.

\section*{Declarations}
\noindent\textbf{Conflict of interest.} The authors declare that they have no conflict of interest.

\end{document}